\documentclass{amsart}

\usepackage{amsmath}
\usepackage{amssymb}
\usepackage{amsthm}
\usepackage{latexsym}
\usepackage{bbold}
\usepackage{stmaryrd}
\usepackage{dsfont}
\usepackage{hyperref}
\usepackage{enumerate}
\usepackage{color}
\usepackage[noend,boxed,linesnumbered,ruled]{algorithm2e}

\newtheorem{theorem}{Theorem}[section]
\newtheorem{corollary}[theorem]{Corollary}
\newtheorem{lemma}[theorem]{Lemma}
\newtheorem{proposition}[theorem]{Proposition}

\theoremstyle{definition}

\newtheorem{remark}[theorem]{Remark}

\newcommand{\integers}{\ensuremath{\mathds{Z}}} 
\newcommand{\rationals}{\ensuremath{\mathds{Q}}} 

\newcommand{\comment}[1]{}

\newcommand{\logspace}{\textsc{\textbf{LSPACE}}}

\newcommand{\eval}[1]{\ensuremath{\mathrm{eval}({#1})}}

\newcommand\malcevgen{lower central Mal'cev } 
\newcommand{\coords}[1]{\ensuremath{\mathrm{Coord}({#1})}} 
\newcommand{\coordsj}[2]{\ensuremath{\mathrm{Coord}_{{#2}}({#1})}} 
\newcommand{\pivot}[1]{\ensuremath{\pi_{{#1}}}} 

\newcommand{\am}[1]{}{}
\newcommand{\sv}[1]{}{}
\newcommand{\jm}[1]{}{}
\newcommand{\an}[1]{}{}

\title{Logspace and compressed-word computations in nilpotent groups}

\author[J.Macdonald]{Jeremy Macdonald}
\address{Concordia University, 1455 De Maisonneuve Blvd. W.,
Montreal, Quebec, Canada,
H3G 1M8}
\curraddr{}
\email{}
\thanks{}

\author[A.Myasnikov]{Alexei Myasnikov}
\address{Stevens Institute of Technology,
Castle Point on Hudson,
Hoboken, NJ 07030}
\curraddr{}
\email{}
\thanks{}

\author[A.Nikolaev]{Andrey Nikolaev}
\address{Stevens Institute of Technology,
Castle Point on Hudson,
Hoboken, NJ 07030}
\curraddr{}
\email{}
\thanks{}

\author[S.Vassileva]{Svetla Vassileva}
\address{Concordia University, 1455 De Maisonneuve Blvd. W.,
Montreal, Quebec, Canada,
H3G 1M8}
\curraddr{}
\email{}
\thanks{}


\date{}

\begin{document}
\maketitle{}
\begin{abstract}
For finitely generated nilpotent groups, we employ Mal'cev coordinates to solve several classical algorithmic problems 
efficiently.  Computation of normal forms, the membership problem, the conjugacy problem, and computation of presentations 
for subgroups are solved using only logarithmic space and quasilinear time. Logarithmic space presentation-uniform versions of these algorithms are provided. Compressed-word versions of the same problems, in which each input word is provided as a straight-line program, are solved in polynomial time.
\end{abstract}

\tableofcontents







\section{Introduction}


Algorithmic properties of finitely generated nilpotent groups have been extensively studied, and many algorithmic problems in these groups are known to be decidable.
However, in most cases neither the computational complexity of the existing algorithms nor the exact complexity of the problems themselves is known. Some of the earlier results in this area were obtained before complexity issues became a concern in algebra, while others were mostly focused on practical aspects of computing.
Lately, ideas from space complexity are exerting a strong influence on computations in algebra, and in particular there is an active research interest in logarithmic space  computations.  Another influence of computer science on modern algebraic computing concerns data representation in compressed form and its use in developing algorithms with lower computational complexity. Polynomial time, logarithmic space, and compressed-word computations are now main players in modern algorithmic group theory. We will elaborate their roles in \S1.2--\S1.4 below. 

In this paper we study the computational complexity of several fundamental algorithmic problems in finitely generated nilpotent groups. These problems include 
computing normal forms (Mal'cev coordinates) of group elements (hence deciding the word problem), deciding the conjugacy and membership problems,  computing kernels of homomorphisms, and finding presentations for finitely generated subgroups. We prove that in a fixed group all these problems are computable in space logarithmic in the size of the input and, simultaneously, in quasilinear time. 

For the decision problems above (i.e., word, conjugacy, and membership problems), our algorithms also solve the `search' version of the problem, meaning if the algorithm answers ``Yes'' it also provides a witness to its answer, in logarithmic space and quasilinear time.
For the word problem, this means writing a trivial element as a product of conjugates of relators; for the conjugacy problem, finding a conjugator; for the membership problem, expressing the given element as a product of subgroup generators. Note that, generally speaking, the search version of the word problem can be arbitrarily more difficult than the decision problem, for example, see~\cite[Section 4]{Madlener-Otto:85}.


We also consider compressed-word versions of 
these problems, in which each input word is provided as a straight-line program (compressed word) producing a word over the generating set.  We solve all compressed versions in time 
quartic in the input size.

All our algorithms \emph{can} be executed 
uniformly, meaning the nilpotent group may be given by an arbitrary presentation as part of the input, but in general this will invalidate the complexity assessment 
as the nilpotency class and number of generators play a role in the complexity bound.  However, if both the number of generators and the nilpotency class are bounded, 
the algorithms run in logarithmic space and polynomial time, with the degree of the polynomial depending on the bound.


\subsection{Known approaches and summary of new results}

An early example of the study of algorithmic problems in nilpotent groups is the work of Mostowski~\cite{Mostowski}, who provided solutions to several algorithmic problems (word, membership, finiteness problems) and further expressed hope that the algorithms are practical for carrying out on a digital computer. As another example, while the word problem in nilpotent groups has long been known to be decidable, it was only established comparatively recently~\cite{HR} that it is in fact decidable in real (therefore, linear) time. In 1958, Mal'cev~\cite{Malcev} investigated finite approximations of finitely generated nilpotent groups, which allowed him to prove that they have decidable membership problem. In 1965, Blackburn~\cite{Blackburn}, using the same method, showed decidability of the conjugacy problem for the same class of groups. Such separability arguments, while sufficient to show that the corresponding decision problems are solvable, offer, by themselves, no reasonable estimates on the time or space complexity of the algorithms involved. Only very recently, certain bounds for so-called full residual finiteness growth of nilpotent groups were established~\cite{BouRabee}.

We would like to point out two more established approaches to algorithmic questions in nilpotent groups. One is based on the fact that every torsion-free nilpotent group embeds into a linear group $\mathrm{UT}_{d}(\integers)$ (see for example~\cite{Hal69} or~\cite{Kargapolov-Merzlyakov}).  This easily solves the word problem, and was used by Grunewald and Segal in 1980~ \cite{GS80} to solve the last (at the time) standing major algorithmic problem in nilpotent groups, the isomorphism problem. It is worth mentioning that such an embedding shows that the word problem for torsion-free finitely generated nilpotent groups is, indeed, in \logspace{}~\cite{Lipton_Zalc}, 
but does not seem to provide any concrete time or space complexity estimates in the case of the conjugacy or membership problems.


Another fruitful approach is due to the fact that finitely generated nilpotent groups admit a so-called Mal'cev (Hall--Mal'cev) basis (see, for example, \cite{Hal69} and \cite{Kargapolov-Merzlyakov}), which allows one to carry out group operations by evaluating polynomials (see Lemma~\ref{le:malcev}). This approach was systematically used in~\cite{KRRRCh}, which gave solutions to a number of algorithmic problems  in a class of groups that includes finitely generated nilpotent groups. It was also used in the above-mentioned work of Mostowski~\cite{Mostowski}. 

Mal'cev basis techiniques can be viewed as part of a more general picture. Indeed, every finitely generated nilpotent group $G$ is polycyclic.  To a polycyclic series one may associate a \emph{polycyclic presentation}, which in the case of nilpotent groups is closely connected to establishing a Mal'cev basis. Many algorithimic problems may be solved using such a presentation. This approach is described in detail in \cite{Sim94} and further studied in~\cite{LGS98, Nickel}, and in particular may be used to solve the membership and conjugacy problems in polycyclic groups. To our knowledge, no robust complexity estimates for such methods have been  established in the case of nilpotent groups, with one exception of the recent papers~\cite{MNU2,MNU3}, which find  polynomial bounds for the equalizer and membership problems.


We follow the Mal'cev basis approach.  Let $G$ be a fixed finitely generated nilpotent 
group. We describe algorithms to solve each of the following problems. 
\begin{enumerate}[(I)]
\item \label{Prob:NF} Given $g\in G$, compute the (Mal'cev) normal form of $g$.
\item \label{Prob:MP} Given $g,h_{1},\ldots,h_{n}\in G$, decide whether
 $g\in \langle h_{1},\ldots,h_{n}\rangle$ and if so express $g$ as a product of subgroup generators from a standardized set.
\item \label{Prob:KP} Fix another finitely generated nilpotent group, $H$.  
Given $K=\langle g_{1},\ldots,g_{n}\rangle$, a homomoprhism $\phi:K\rightarrow H$ specified 
by $\phi(g_{i})=h_{i}$, and $h\in \mathrm{Im}(\phi)$, compute a generating set for $\ker(\phi)$ and find $g\in G$ such that $\phi(g)=h$.
\item \label{Prob:P} Given $K=\langle g_{1},\ldots,g_{n}\rangle \leq G$, compute a presentation for $K$.
\item \label{Prob:C} Given $g\in G$, compute a generating set for the centralizer of $g$.
\item \label{Prob:CP} Given $g,h\in G$, decide whether or not there exists $u\in G$ such that $u^{-1}gu=h$ and if so find such an element $u$.
\end{enumerate}
In every case, 
the algorithm runs in space $O(\log L)$ and, simultaneously, in time $O(L\log^{3} L)$, where $L$ is the size of the input of the given problem. Problems (\ref{Prob:NF}), 
(\ref{Prob:C}), and (\ref{Prob:CP}) 
in fact run in time $O(L\log^{2} L)$. 
To every subgroup one may associate a standardized or 
\emph{full-form} generating set, and it is this set that is 
used in problem (\ref{Prob:MP}) to express $g$ 
as a product of subgroup generators; we may in addition express $g$ in 
terms of the given generating set, but the algorithm then 
runs in polynomial time (and not logspace).

A Mal'cev basis consisting of $m$ elements establishes a coordinatization of $G$, whereby each element is identified with an $m$-tuple of integers.  The coordinates of 
a product $gh$ are given by polynomial functions of the coordinates of $g$ and $h$,
and the coordinates of a power $g^l$ are given by polynomial functions of $l$ and 
the coordinates of $g$.
Our algorithms work directly with the coordinate tuples and these multiplication/exponentiation 
polynomials.  
The key to obtaining 
logarithmic space bounds, and polynomial time bounds for compressed-word problems, is the fact that coordinates of an $n$-fold product 
$g_{1}g_{2}\ldots g_{n}$ are bounded in magnitude by a polynomial function whose degree 
depends only on the nilpotency class $c$ of $G$ (a constant) and not, as may be inferred by composing the polynomials $n$ times, on the length of the product (Theorem~\ref{Lem:PolyCoords}).

The class \logspace{} of problems decidable in logarithmic space is defined via machines called \emph{logspace transducers}, 
and we recall the relevant definitions in \S\ref{Section:Logspace}.  
Logarithmic space computations in groups have been studied primarily in relation to the word and normal form problems.   
In free groups, the word problem was solved in logarithmic space by~\cite{Lipton_Zalc}.
Normal forms were computed in logarithmic space for graph groups and Coxeter groups in~\cite{DKL} and the class of groups with logspace-computable normal forms was shown to be closed under several important constructions in~\cite{EEO}.  The conjugacy problem was also solved in logspace for free solvable groups, the Grigorchuk group, and certain wreath products in~\cite{Svetla}.

We also consider compressed-word versions of problems (\ref{Prob:NF})--(\ref{Prob:CP}). 
In this formulation, every input word is given in the form of a 
straight-line program (see \S\ref{Section:SLP}) and the input size $L$ is the sum of the sizes of all input programs. 
A program of size $L$ may encode a word of length up to $2^{L-1}$, and so 
efficient (i.e. polynomial time)  algorithms must work directly on the straight-line 
program without producing the encoded word itself. 
We solve all of the problems (\ref{Prob:NF})--(\ref{Prob:CP}), with compressed-word inputs, in 
time $O(L^{4})$, with  (\ref{Prob:NF}), (\ref{Prob:C}), and (\ref{Prob:CP}) being 
solved in time $O(L^{3})$. The approach is to first solve (\ref{Prob:NF}) to compute the Mal'cev coordinates of each input element, write each coordinate as a binary number, then 
apply the previous algorithms to the coordinate tuples.  This also shows that 
a second `compressed' formulation, in which 
every input word is given by its (binary) Mal'cev coordinates, can be solved in 
the same time complexity for each problem.

The compressed version of the word problem is known to be polynomial-time decidable in several classes of groups, including free groups \cite{Loh04},  
partially commutative groups \cite{LS07}, limit groups \cite{Mac10}, and nilpotent groups \cite{HLM12}; further, 
polynomial-time decidability is preserved under many important 
group-theoretic constructions, such as graph products (see \cite{Loh12} for a summary). 
One motivation for obtaining a polynomial-time solution to the compressed word problem 
in a group $G$ is that such a solution give rise to a polynomial-time solution to the (non-compressed) 
word problem in any finitely generated subgroup of $\mathrm{Aut}(G)$ 
and in semi-direct products involving $G$ \cite{Sch08}.

Less is known about the compressed conjugacy problem. It is poly\-no\-mial-time decidable in free groups~\cite{Sch08} and more generally partially commutative groups~\cite{HLM12}, these results being part of a polynomial-time solution to the word problem in the outer automorphism group of these groups. Further, the compressed conjugacy problem was recently shown to be polynomial-time decidable in hyperbolic groups~\cite{Holt-Lohrey-Schleimer:2018}. For the compressed membership problem, we are not aware of any previous results for interesting classes of groups. Even in free groups no polynomial-time algorithm is known, though recent results of Jez \cite{Jez12} on DFAs with compressed labels are closely related. Compressed membership in abelian groups is easily solved in polynomial time by converting the straight-line programs to $O(n)$-bit integer vectors and applying linear algebra techniques, and our proof for nilpotent groups uses a similar approach. 


\subsection{Logspace}\label{Section:Logspace}
We define logarithmic space computation via a machine called a \emph{logspace transducer}. 
Briefly, this is a deterministic Turing machine with three tapes: a read-only input tape, a read-write work tape with number of cells 
logarithmic in the size of the input tape, and a write-only output tape.  We provide the details below.


Let $\Sigma$ be a finite alphabet containing a symbol $\epsilon$ called the \emph{blank symbol}. 
A \emph{tape} is an infinite sequence $X=\{x_n\}_n$ with $x_n\in \Sigma$ and all but finitely many $x_{n}$ being the blank symbol. 
The subsequence consisting of all symbols up to and including the last non-blank symbol is called the \emph{content} of the tape and the length 
of this sequence is the \emph{size} of the tape. 
Intuitively, we think of it as a one-ended infinite array of cells, each cell holding an element of this sequence. 
To every tape $X$ we associate a positive integer $h_X$ called the \emph{head position}.


Let $S$ be a finite set, called the set of \emph{states}. 
A \emph{configuration} is a tuple $C=(s,I,h_{I}, W, h_{W}, O, h_{O})$ consisting of a state 
$s\in S$ and three tapes $I,W,O$ called the \emph{input tape}, \emph{work tape}, and 
\emph{output tape} (respectively) together with the head positions for each tape.  
A \emph{transducer} is a function which assigns to every possible configuration 
$C$ a successor configuration $C'=(s',I',h_{I'}, W', h_{W'}, O', h_{O'})$ with the following properties: 
\begin{itemize}
\item $C'$ depends only on $s$ and the symbols at $h_{I}$ on $I$ and at $h_{W}$ on $W$;
\item $I=I'$ and $h_{I'}$ differs from $h_{I}$ by at most 1; 
\item $W$ and $W'$ are identical except possibly at position $h_{W}$, and $h_{W'}$ differs from $h_{W}$ by at most 1;
\item either $h_{O}=h_{O'}$ and $O=O'$ or $h_{O'}=h_{O}+1$ and $O'$ differs from $O$ only in position $h_{O}$.
\end{itemize} 
A \emph{run} of the transducer is a finite sequence of configurations 
$C_{1}, C_{2}, \ldots, C_{k}$ where $C_{i+1}=C_{i}'$ for $i=1,\ldots,k-1$, 
$C_{k}'=C_{k}$ (there is no further computation to perform), and the work and output tapes of $C_{1}$ contain only 
the blank symbol $\beta$.  The content of the input tape of $C_{1}$ is called the \emph{input} and the content of the output tape 
of $C_{k}$ is the \emph{output}. 

Let $c>0$ be any integer.  A transducer is called a \emph{$c$-logspace transducer} if for every possible run the size of the work tape in 
every configuration is bounded by $c\log n$ where $n$ is the size of the input tape 
(the base of the logarithm is not generally relevant, though using $|\Sigma|$ or 
$|\Sigma|-1$ is natural).  
Provided such a constant $c$ exists for a given transducer, it will be called simply a \emph{logspace transducer}. 

Let $\Sigma'$ be another alphabet. We say that a function 
$f: \Sigma^*\rightarrow (\Sigma')^*$ is \emph{logspace computable} if there is a logspace transducer which for every input $w\in \Sigma^*$ produces $f(w)$ on the output tape. A \emph{decision problem}, which we define as a subset of $\Sigma^*$, is \emph{logspace decidable} if its characteristic function is logspace computable. 
The complexity class \logspace{} consists of all decision problems that are logspace 
decidable. 
Note that in order to discuss multi-variable functions, we simply add a new symbol $\alpha$ to the alphabet and separate the input words by this symbol. 

Any function that is logspace computable is also computable in polynomial time 
(meaning the length of every run is bounded by a polynomial function of $n$). 
Indeed, in any run the sequence of configurations that follow a given configuration $C_{i}$ are determined by $C_{i}$ only.  Hence no run may contain the same configuration twice since runs have finite length by definition. 
Thus, the length of any given run (and hence the time complexity of the machine) is bounded by the number of possible configurations 
\[|S|\cdot n  \cdot |\Sigma|^{c\log n}\cdot c\log n \in O(n^{c+2}),
\] 
where $n$ is the length of the input. Since the degree $c+2$ of this polynomial 
can be expected to be quite high, it is usually advantageous to analyze the 
time complexity of logspace transducers directly and obtain, if possible, a 
low-degree polynomial time bound.

The type of computations that can be done in logspace are quite limited. 
For example, to store an integer $M$ requires $\log M$ bits.  Hence we can store and 
count up to $n^{c}$, but not higher.  We may also store and manipulate pointers to 
different locations in the input, as each such pointer is an integer of size at 
most $n$.  Basic arithmetic operations are also computable in logspace, and it 
is possible to compose logspace transducers (i.e., the class of logspace computable functions is closed under composition).  

The above is a formal description of logspace computability via transducers, but in practice we simply work with informal algorithm descriptions and ensure that 
our algorithms require no more than logarithmic space.  Each of our algorithms may 
be formally encoded as a logspace transducer.


\subsection{Compressed words}\label{Section:SLP}
Let $\Sigma$ be a set of symbols containing a special symbol $\epsilon$ used to 
denote the empty word. 
A \emph{straight-line grammar $\mathds{A}$ over $\Sigma$ in Chomsky normal form} consists of an ordered finite set 
$\mathcal{A}$, called the set of \emph{non-terminal symbols}, together with for each $A\in \mathcal{A}$ exactly one \emph{production rule} either of the form 
$A \rightarrow BC$ where $B,C\in \mathcal{A}$ and $B,C<A$ or of the form $A \rightarrow x$ where $x\in \Sigma$.
The greatest non-terminal is called the \emph{root}, elements of $\Sigma$ are called \emph{terminal symbols}, and the \emph{size} 
of $\mathds{A}$ is the number of non-terminal symbols and is denoted $|\mathds{A}|$. For shortness, with a slight abuse of terminology, in this paper we refer to straight-line grammars in Chomsky normal form as \emph{straight-line programs} or \emph{compressed words}.
Note that any program $\mathds{A}$ may, by encoding the non-terminal symbols as integers, be written down using $O(|\mathds{A}|\cdot\log|\mathds{A}|)$ bits. 

The \emph{output} or \emph{evaluation} of $\mathds{A}$ is the word in $\Sigma^{*}$ obtained by starting with the root non-terminal and 
successively replacing every non-terminal symbol with the right-hand side of its production rule.  It is denoted $\eval{\mathds{A}}$ and  
we similarly denote by $\eval{A}$ the word obtained starting with the non-terminal $A\in \mathcal{A}$. 
For example, the program $\mathds{B}$ over $\{x\}$ with production rules 
\[
B_{n}\rightarrow B_{n-1}B_{n-1}, \;B_{n-1}\rightarrow B_{n-2}B_{n-2}, \;\ldots,\; B_{1}\rightarrow x
\]
has $\eval{\mathds{B}}=x^{2^{n-1}}$ and $\eval{B_{2}}=x^{2}$. 
As this example illustrates, the length of $\eval{\mathds{A}}$ may be exponential in $|\mathds{A}|$ (this program 
in fact achieves the maximum-length output).
Thus to have efficient algorithms dealing with compressed words, one must avoid computing $\eval{\mathds{A}}$ and instead work directly 
with the production rules.  A fundamental result of Plandowski \cite{Pla94} states that two straight-line programs may be checked 
for character-for-character equality of their outputs in time polynomial in the sum of the program lengths.  

Let us note here that for any fixed word $w=x_{1}x_{2}\cdots x_{m}$ over $\Sigma$, the word $w^{n}$ may be encoded with a program of size $O(\log(n))$.  
Indeed, we may first encode $w$ using binary subdivision, i.e. with the scheme $W_{i}\rightarrow x_{i}$ for $i=1,\ldots,m$, $W_{m+1}\rightarrow W_{1}W_{2}$, 
$W_{m+2}\rightarrow W_{3}W_{4}$ and so on, obtaining $\eval{W_{k}}=w$ where $k$ is 
some integer bounded by $|w|+\log_{2}|w|$.  
A program similar to $\mathds{B}$ above, with $W_{k}$ in place of $x$ and 
suitable modifications when $n$ is not a power of 2, will encode $w^{n}$.

A straight-line program over a group $G$ is a straight-line program over a given symmetrized generating set of $G$.  Any algorithmic problem 
for $G$ which takes words as input may be considered in `compressed-word form' where all input words are provided as straight-line programs. 
For example, the compressed conjugacy problem asks, given two straight-line programs $\mathds{A}$ and $\mathds{B}$ over $G$, whether or not 
$\eval{\mathds{A}}$ and $\eval{\mathds{B}}$ represent conjugate group elements.

\section{Computing Mal'cev normal forms}

To produce efficient algorithms, our nilpotent group will need to be given by a 
particular type of presentation, known as a consistent nilpotent presentation. 
Such a presentation can be computed from an arbitrary presentation 
(Prop. \ref{pr:find_malcev}).  During computation, we represent group 
elements in their Mal'cev normal form, which we define below.  Critically, 
converting a general group word to Mal'cev form involves at most a polynomial 
expansion in word length (Theorem \ref{Lem:PolyCoords}), 
and may be performed in logarithmic space (Theorem \ref{th:compute_malcev}).

\subsection{Nilpotent groups and Mal'cev coordinates}\label{Section:Nilpotent}

A group $G$ is called \emph{nilpotent} if it possesses  central series, i.e. a normal series 
\begin{equation}\label{Eqn:CyclicSeries}
G=G_{1}\rhd G_{2}\rhd \ldots \rhd G_{s}\rhd G_{s+1}=1
\end{equation}
such that $G_{i}/G_{i+1}\leq Z(G/G_{i+1})$ for all $i=1,\ldots,s$ where $Z$ denotes the center.  Equivalently, $G$ possesses  a normal series in which $[G,G_{i}]\leq G_{i+1}$ for all $i=1,\ldots,s$.

If $G$ is finitely generated, so are the abelian quotients $G_i/G_{i+1}$, $1\le i\le s$. Let $a_{i1},\ldots,a_{im_i}$ be a standard basis of $G_i/G_{i+1}$, i.e. a generating set in which $G_i/G_{i+1}$ has presentation $\langle a_{i1},\ldots,a_{im_i}| a_{ij}^{e_{ij}}, j\in \mathcal T_i\rangle$ in the class of abelian groups, where 
$\mathcal{T}_{i}\subseteq\{1,\ldots,m_{i}\}$ and $e_{ij}\in\integers_{>0}$. Formally put $e_{ij}=\infty$ for $j\notin \mathcal T_i$. Note that
\[
A=\{a_{11},a_{12},\ldots,a_{sm_s}\}
\]
is a polycyclic generating set for $G$, and we call $A$ a {\em Mal'cev basis associated to the central series}~\eqref{Eqn:CyclicSeries}. 

For convenience, we will also use a simplified notation, in which the generators $a_{ij}$ and exponents $e_{ij}$ are renumbered by replacing each subscript $ij$ with $j+\sum\limits_{\ell<j}m_\ell$, so the generating set $A$ can be written as $A=\{a_1,\ldots, a_m\}$. We allow the expression $ij$ to stand for  
$j+\sum\limits_{\ell<j}m_\ell$ in other notations as well.
We also denote
\[
\mathcal{T} = \{i \;|\; e_{i}<\infty\}.
\]
By the choice of the set $\{a_{1},\ldots,a_{m}\}$ every element $g\in G$ may be written uniquely as a group word of the form
\[
g=a_{1}^{\alpha_{1}}\ldots a_{m}^{\alpha_{m}},
\]
where $\alpha_{i}\in \integers$ and $0\leq \alpha_{i}<e_{i}$ whenever $i\in \mathcal{T}$.  
The $m$-tuple $(\alpha_{1},\ldots,\alpha_{m})$ is called the \emph{coordinate tuple} 
of $g$ and is denoted $\coords{g}$, and the expression $a_{1}^{\alpha_{1}}\ldots a_{m}^{\alpha_{m}}$ is called the \emph{(Mal'cev) normal form} of $g$. We also 
denote $\alpha_{i}=\coordsj{g}{i}$. When we need to make a distinction between integers and their binary notation for algotithmic purposes, we refer to the tuple of binary notations for $(\alpha_{1},\ldots,\alpha_{m})$ as \emph{binary coordinate tuple} of $g$.

To a Mal'cev basis $A$ we associate a presentation of $G$ as follows. 
For each $1\le i\le m$, let $n_i$ be such that $a_i\in G_{n_i}\setminus G_{n_i+1}$. If $i\in \mathcal{T}$, then $a_{i}^{e_{i}}\in G_{n_i+1}$, hence a relation 
\begin{equation}\label{stdpolycyclic1}
a_{i}^{e_{i}} = a_{\ell}^{\mu_{i\ell}}\cdots a_{m}^{\mu_{im}}
\end{equation}
holds in $G$ for some $\mu_{ij}\in\integers$ and $\ell>i$ such that $a_\ell,\ldots,a_m\in G_{n_i+1}$.  Let $1\leq i<j\leq m$.  Since the series~\eqref{Eqn:CyclicSeries} is central, 
relations of the form 
\begin{eqnarray}
a_{j}a_{i} & = &  a_{i} a_{j}a_{\ell}^{\alpha_{ij\ell}}\cdots a_{m}^{\alpha_{ijm}} \label{stdpolycyclic2}\\
a_{j}^{-1}a_{i} & = & a_{i} a_{j}^{-1} a_{\ell}^{\beta_{ij\ell}}\cdots a_{m}^{\beta_{ijm}} \label{stdpolycyclic3}
\end{eqnarray}
hold in $G$ for some $\alpha_{ijk},\beta_{ijk}\in\integers$ and $l>j$ such that $a_\ell,\ldots, a_m\in G_{n_j+1}$.  

The set of (abstract) symbols $\{a_{1},\ldots,a_{m}\}$ together with relators (\ref{stdpolycyclic1})--(\ref{stdpolycyclic3}) present a group $G'$ that is isomorphic to $G$ 
under the natural isomorphism (relator (\ref{stdpolycyclic3}) may be omitted when $j\in\mathcal{T}$).  
Indeed, \emph{any} presentation 
on symbols $\{a_{1},\ldots,a_{m}\}$ with relators of the form (\ref{stdpolycyclic1}) for any choice of $\mathcal{T}$ 
and relators of the form (\ref{stdpolycyclic2}) and (\ref{stdpolycyclic3}) for all $1\leq i<j\leq m$ defines a 
nilpotent group $G''$ with cyclic central series having terms $\langle a_{i},\ldots,a_{m}\rangle$ for $i=1,\ldots,m$.  Such a presentation is called a \emph{consistent nilpotent presentation for $G''$} if the order of $a_{i}$ 
modulo $\langle a_{i+1},\ldots,a_{m}\rangle$ is precisely $e_{i}$. While presentations of this 
form need not, in general, be consistent, those derived from a central series of a group $G$ as above must be consistent
(if in $G'$ the order of $a_{i}$ is $e_{i}'<e_{i}$, then this fact follows from 
(\ref{stdpolycyclic1})--(\ref{stdpolycyclic3}) and hence a relation for $a_{i}^{e_{i}'}$ would have been written in 
(\ref{stdpolycyclic1})). Consistency of the presentation implies that $G'\simeq G$.



\begin{proposition}\label{pr:find_malcev}
There is an algorithm that, given a finite presentation of a nilpotent group $G$, finds a consistent nilpotent presentation of $G$ and an explicit isomorphism.  The presentation 
may be chosen to be the presentation derived from a Mal'cev basis associated 
with the lower or upper central series of $G$.
\end{proposition}
\begin{proof}
Prop. 3.2 of \cite{BCR91} proves that a given nilpotent presentation may be checked for consistency, so to produce a consistent nilpotent presentation
it suffices to enumerate 
presentations obtained from the given presentation of $G$ by finite sequences of Tietze transformations until a consistent nilpotent 
presentation is found (cf. \cite{BCR91} Thm. 3.3).

Further, we may compute both the lower and upper central series from a presentation of $G$. Indeed, each term $\Gamma_{i}=[G,\Gamma_{i-1}]$ of the lower central series is 
precisely the normal closure of the set $\{[g,\gamma]\}_{g,\gamma}$ where $g$ and $\gamma$ 
run over generating sets of $G$ and $\Gamma_{i-1}$, respectively, and as such a generating 
set may be 
computed by \cite{BCR91} Lem. 2.5. The upper central series is computed in Cor. 5.3 of the 
same work.  From either series we may find a Mal'cev basis and then write down, via 
exhaustive search, the required relators (\ref{stdpolycyclic1})-(\ref{stdpolycyclic3}).
\end{proof}

In Proposition~\ref{pr:find_presentation_poly} we employ techniques described in Section~\ref{se:matrix_reduction} to give a polynomial-time version of Proposition~\ref{pr:find_malcev}.

An essential feature of the coordinate tuples for nilpotent groups is that the coordinates of a product 
$(a_{1}^{\alpha_{1}}\cdots a_{m}^{\alpha_{m}})(a_{1}^{\beta_{1}}\cdots a_{m}^{\beta_{m}})$ may be computed as a ``nice'' (polynomial if $\mathcal{T}=\emptyset$) function 
of the integers $\alpha_{1},\ldots,\alpha_{m},\beta_{1},\ldots,\beta_{m}$, and the coordinates of a power 
$(a_{1}^{\alpha_{1}}\cdots a_{m}^{\alpha_{m}})^{l}$ may similarly be computed as a `nice' function of $\alpha_{1},\ldots,\alpha_{m}$ and $l$.  The existence of such polynomial functions for torsion-free nilpotent groups is proved in \cite{Hal69} and \cite{Kargapolov-Merzlyakov}, 
and an explicit algorithm to construct them from a nilpotent presentation of $G$ 
is given in \cite{LGS98}.

If any of the factors $G_{i}/G_{i+1}$ are finite (which must occur when $G$ has torsion), the coordinate functions also involve the extraction 
of quotients and remainders modulo $e_{i}$.
For each $i\in \mathcal{T}$, we define functions $r_{i}:\integers\rightarrow\{0,1,\ldots,e_{i}-1\}$ and 
$s_{i}:\integers\rightarrow\integers$ by the 
decomposition 
\[
t = r_{i}(t) + s_{i}(t) e_{i},
\]
where $t\in \integers$.  Let $\mathcal{F}_{\rationals}(n)$ denote the set of all functions 
$f:\integers^{n}\rightarrow\integers$ 
formed as a finite composition of the functions from the set $\{\cdot,+,r_{i},s_{i}\,|\, i\in \mathcal{T}\}$ using constants from $\rationals$, 
where $\cdot$ and $+$ denote multiplication and addition in $\rationals$.

\begin{lemma}\label{le:malcev}
Let $G$ be a nilpotent group with Mal'cev basis $a_{1},\ldots,a_{m}$. 
There exist $p_{1},\ldots,p_{m}\in \mathcal{F}_{\rationals}(2m)$ and 
$q_{1},\ldots,q_{m}\in\mathcal{F}_{\rationals}(m+1)$ satisfying the following properties.  For every $g,h\in G$ and $l\in\integers$, 
writing $\coords{g}=(\gamma_{1},\ldots,\gamma_{m})$ and $\coords{h}=(\delta_{1},\ldots,\delta_{m})$, 
\begin{enumerate}[(i)]
\item $\coordsj{gh}{i} = p_{i}(\gamma_{1},\ldots,\gamma_{m},\delta_{1},\ldots,\delta_{m})$,
\item $\coordsj{g^{l}}{i} = q_{i}(\gamma_{1},\ldots,\gamma_{m}, l)$,
\item if $i\in \mathcal{T}$ then $p_{i} = r_{i}\circ p_{i}'$ and $q_{i}=r_{i}\circ q_{i}'$ for some $p_{i}'\in \mathcal{F}_{\rationals}(2m)$ and 
$q_{i}'\in \mathcal{F}_{\rationals}(m+1)$, and 
\item\label{le:ab_operation} if, for some index $1\le k<m$, $\gamma_{i}=0$ for all $i\leq k$, or $\delta_{i}=0$ for all $i\leq k$, then for all $i\leq k+1$
	\begin{enumerate}[(a)]
	\item $\coordsj{gh}{i}=\gamma_{i}+\delta_{i}$ and $\coordsj{g^{l}}{i}=l \gamma_{i}$ if $i \not\in\mathcal{T}$, and 
	\item $\coordsj{gh}{i}=r_{i}(\gamma_{i}+\delta_{i})$ and $\coordsj{g^{l}}{i}=r_{i}(l \gamma_{i})$ if $i\in\mathcal{T}$.
	\end{enumerate} 
\end{enumerate}
\end{lemma}

\begin{proof} This lemma is a special case of Theorem~6.7 of~\cite{MS13}.
\end{proof}

We note that the existence of such functions extends to nilpotent groups admitting exponents in a binomial principal ideal domain, as described in Theorem~5.7 of \cite{MS13}.

Computation of the functions $p_1,\ldots,p_m$, $q_1,\ldots,q_m$, in the case 
when all factors $G_{i}/G_{i+1}$ are infinite, can be done via the ``Deep Thought'' algorithm~\cite{LGS98}. 
If some factors are finite, one may introduce the functions $r_{i}$ and $t_{i}$ 
one factor at a time 
using a procedure similar to that used to compute normal forms given in \S 7.2 of~\cite{LGS98}. It is worth observing that in this construction the values of the numbers $\alpha_{ijl}$, $\beta_{ijl}$, $\mu_{il}$, and $e_{i}$ are not essential: 
one may compute functions in which all of these values appear as variables.


For \S\ref{Section:MatrixReduction} and \S\ref{Section:Homomorphisms}, and the 
remainder of this section, we fix $G$ to be a finitely presented nilpotent group.  
Set $c$ to be the nilpotency class of $G$ and fix a Mal'cev basis $A=\{a_{1},\ldots,a_{m}\}$ associated with 
the lower central series of $G$, with $m$ the size of this basis. 
Algorithmic results in these sections do not take $G$ as part of the input and so we may, 
in light of Proposition \ref{pr:find_malcev} and 
Lemma \ref{le:malcev}, assume that the presentation of $G$ is precisely the 
consistent nilpotent presentation corresponding to $A$ 
and that the functions $p_{i}$, $q_{i}$ are known (computed in advance).  
However, we state algorithmic results without restriction on the presentation of $G$ with the 
understanding that such algorithms will, in pre-computation, translate to such a presentation 
if needed.
In \S\ref{Section:Uniform} 
we provide uniform algorithms, in which $G$ is included in the input.

\subsection{Polynomial bound on the length of normal forms}

Suppose $w=x_{1}x_{2}\cdots x_{n}$ is a word over $A^{\pm}$.  In order to compute the coordinate tuple of $w$, 
one may use the polynomials $p_{i}$ to compute the coordinates of $x_{1}x_{2}$, then use this result to find the coordinates of 
$(x_{1}x_{2})x_{3}$ and 
so on.  However, the resulting computation is an $n$-fold composition of the polynomials and thus we may a priori expect a bound of order $k^{n}$ on the magnitude of the 
coordinates, with $k$ being the maximum degree of the polynomials. This presents an obstacle to logspace computation, as the binary representation of integers that size 
requires linear space.  We show that the coordinates are in fact of order $n^{c}$, where $c$ is the nilpotency class of $G$. The following result is, to our knowledge, folklore (our proof is adapted from unpublished lecture notes by C.Dru\c{t}u and M.Kapovich).



\begin{theorem}\label{Lem:PolyCoords}\label{le:poly_coords}
Let $G$ be a nilpotent group with a lower central series $G=\Gamma_1\rhd \ldots\rhd \Gamma_c\rhd \Gamma_{c+1}=\{1\}$ 
with an associated Mal'cev basis  $A=\{a_{11}, \ldots, a_{cm_c}\}$.
There is a constant $\kappa$, depending only on the
presentation of $G$, such that for every word $w$ over $A^{\pm}$,
%
\begin{equation}\label{Eqn:PolynomialCoordinateBound}
|\coordsj{w}{ij}| \leq \kappa \cdot |w|^{i}
\end{equation}
for all $i=1,\ldots, c$, $1\le j\le m_i$.
\end{theorem}
\begin{proof}
We must show that if
$$a_{11}^{\gamma_{11}}\cdots a_{ij}^{\gamma_{ij}}\cdots a_{cm_c}^{\gamma_{cm_c}},\quad 1\le i\le c,\ 1\le j\le m_i,
$$
is the normal form of $w$, then $|\gamma_{ij}|\le \kappa|w|^i$.

Note that since $[\Gamma_i,\Gamma_j]\le \Gamma_{i+j}$, we have for each 
$a=a_{ik}^{\pm 1}$ and $a'=a_{j\ell}^{\pm 1}$, $1\le i, j\le c$, $1\le k\le m_i$, $1\le \ell\le m_j$,
a commutation relation of the form
\begin{equation}\label{eq:comm_class}
[a,a']= a_{t 1}^{\alpha_{t 1}}a_{t 2}^{\alpha_{t 2}}\cdots a_{c m_c}^{\alpha_{c m_c}},
\end{equation}
where $t\ge i+j$.
Similarly, for each $a=a_{ik}^{\pm 1}$ with $ik\in\mathcal{T}$, we have a relation 
\begin{equation}\label{eq:torsion_class}
a^{e_{ik}}=a_{t 1}^{\mu_{t 1}}a_{t 2}^{\mu_{t 2}}\cdots a_{c m_c}^{\mu_{c m_c}},
\end{equation}
where $t> i$. Put $E=\max\{e_{ik}| ik\in\mathcal{T}\}$ (or $E=1$ if $\mathcal{T}=\emptyset$). Let $C_0\in \mathbb Z$ be greater than the word length of  the right hand sides of all equalities~\eqref{eq:comm_class}, \eqref{eq:torsion_class}. Note that $C_0$ only depends on the presentation of $G$.

For any word $v$ over $A^{\pm}$ and integer $1\leq n\leq c$, 
denote by $|v|_n$ the number of occurrences of letters $a_{11},\ldots, a_{1m_1},\ldots,a_{n1},\ldots,a_{nm_n}$ and their inverses in $v$. We also formally put $|v|_n=0$ for $n\le 0$ and $|v|_{n}=|v|_{c}$ for $n>c$.

{\bf Claim.} For every $1\le k\le c$ there is a constant $C_{k+1}$ that depends only on $k$ and the presentation of $G$ such that for every word $w$ over $A^{\pm}$ the 
corresponding group element can be represented in the form 
\[
w=_{G} a_{11}^{\gamma_{11}}\cdots a_{ij}^{\gamma_{ij}}\cdots a_{km_k}^{\gamma_{km_k}}\cdot w_{k+1},
\]
where
\begin{enumerate}[(A)]
\item $1\le i\le k,\ 1\le j\le m_i$,
\item $|\gamma_{ij}|\le C_{k+1}|w|^i$, and 
\item $0\le \gamma_{ij}< e_{ij}$ if $ij\in\mathcal{T}$, and 
\item $w_{k+1}\in \Gamma_{k+1}$ is a word in $a_{(k+1)1},\ldots, a_{cm_c}$ with $|w_{k+1}|_{k+\ell}\le C_{k+1}|w|^{k+\ell}$ for all $k+1\le k+\ell\le c$.
\end{enumerate}

We prove this claim by induction on $k$. We allow $k=0$ as the base case, 
which holds with $w_{1}=w$ and $C_{1}=1$.

Suppose the claim holds for some $k-1\ge 0$. 
Denote $w^{(0)}_{k+1}=w_{k}$ and push an occurrence of $a_{k1}^{\pm 1}$ in 
this word to the left, using the commutation relations~\eqref{eq:comm_class}. This gives the expression
$$
w^{(0)}_{k+1} =_{G} a_{k1}^{\pm 1}\cdot w^{(1)}_{k+1}.
$$
Notice that for $a=a_{k1}^{\pm 1}$ and $a'=a_{ij}^{\pm 1}$, the right-hand side 
$R$ of \eqref{eq:comm_class} satisfies $|R|_{k+\ell}=0$ for all $i>\ell$ and 
$|R|_{k+\ell}\leq C_{0}$ otherwise.  Therefore swapping $a_{k1}^{\pm 1}$ with $a_{ij}^{\pm 1}$ 
increases the word length, by at most $C_{0}$, only for $i=1,\ldots,\ell$.  So 
$$
|w^{(1)}_{k+1}|_{k+\ell}\le |w^{(0)}_{k+1}|_{k+\ell}+C_0|w^{(0)}_{k+1}|_{\ell}.
$$
Notice that this inequality holds in fact for all $\ell\in\integers$. 

We proceed in the same fashion to move left all occurrences of $a_{k1}^{\pm 1}$, followed 
by all occurrences of $a_{k2}^{\pm 1}$ and so on.  At step $j+1$ we represent $w_{k+1}^{(j)}=_{G} a^{\pm 1}_{ki_j}\cdot w^{(j+1)}_{k+1}$, where $\{i_j\}$ is a non-decreasing sequence.
We similarly get
\begin{equation}\label{eq:pascal_triangle}
|w^{(j+1)}_{k+1}|_{k+\ell}\le |w^{(j)}_{k+1}|_{k+\ell}+C_0|w^{(j)}_{k+1}|_{\ell},
\end{equation}
for all $\ell\in \integers$.
All letters $a_{ki}^{\pm 1}$ are collected on the left in at most $N\le |w_k|_{k}\le C_k|w|^{k}$ steps, which gives
\begin{equation}\label{eq:wk_before_torsion}
w_{k}=_{G} a_{k1}^{\gamma_{k1}}\cdots a_{km_k}^{\gamma_{km_k}}w^{(N)}_{k+1}
\end{equation}
with $w^{(N)}_{k+1}\in \Gamma_{k+1}$.

We immediately see by the induction hypothesis that $|\gamma_{ki}|\le |w_k|_k\le C_{k}|w|^k$, which delivers (A) provided $C_{k+1}$ is chosen to be at least $C_{k}$. 
We also find a bound on $|w_{k+1}^{(N)}|_{k+\ell}$, for all 
$k+1\leq k+\ell\leq c$, which will be used to prove (C),
by applying 
\eqref{eq:pascal_triangle} repeatedly. Namely, when we apply \eqref{eq:pascal_triangle} $j$ times, $1\le j\le N$, we obtain
\begin{eqnarray*}
|w_{k+1}^{(N)}|_{k+\ell} & \le & |w_{k+1}^{(N-1)}|_{k+\ell}+C_0 |w_{k+1}^{(N-1)}|_\ell \\ 
&\le&
\left(|w_{k+1}^{(N-2)}|_{k+\ell}+C_0 |w_{k+1}^{(N-2)}|_\ell\right)+C_0\left(|w_{k+1}^{(N-2)}|_{\ell}+C_0 |w_{k+1}^{(N-2)}|_{\ell-k}
\right) \\
&=& 
|w_{k+1}^{(N-2)}|_{k+\ell}+2C_0 |w_{k+1}^{(N-2)}|_\ell+C_0^2 |w_{k+1}^{(N-2)}|_{\ell-k}\\
&\le& 
\ldots \\
&\le& 
\genfrac(){0pt}{0}{j}{0} |w_{k+1}^{(N-j)}|_{k+\ell}+ 
\ldots+\genfrac(){0pt}{0}{j}{j}C_0^j|w_{k+1}^{(N-j)}|_{k+\ell-jk} \\
&=& 
\sum_{\iota\le\frac{\ell}{k}}
\genfrac(){0pt}{0}{j}{\iota}C_0^\iota|w_{k+1}^{(N-j)}|_{k+\ell-\iota k}.
\end{eqnarray*}
For $j=N$, this yields
\begin{eqnarray*}
|w_{k+1}^{(N)}|_{k+\ell} & \le &
C_0^{c/k}\sum_{\iota\le\frac{\ell}{k}}
\genfrac(){0pt}{0}{N}{\iota}|w_{k+1}^{(0)}|_{k+\ell-\iota k} =
C_0^{c/k}\sum_{\iota\le\frac{\ell}{k}}\genfrac(){0pt}{0}{N}{\iota}|w_k|_{k+\ell-\iota k}\\
&\le& 
C_0^{c/k}\sum_{\iota\le\frac{\ell}{k}}N^\iota C_k|w|^{k+\ell-\iota k}.
\end{eqnarray*}
Now we recall that $N\le C^k|w|^k$, which gives
\begin{eqnarray*}
|w_{k+1}^{(N)}|_{k+\ell} &\le& 
C_0^{c/k}\sum_{\iota\le\frac{\ell}{k}}C_k^\iota|w|^{\iota k} C_k|w|^{k+\ell-\iota k} \\ 
&\le&
\widehat{C}_{k+1}|w|^{k+\ell}.
\end{eqnarray*}

Before obtaining (C), we must first reduce certain exponents to obtain (B), then repeat 
the collection process described above. Consider the word $a_{k1}^{\gamma_{k1}}\cdots a_{km_k}^{\gamma_{km_k}}$ and, for each $ki\in\mathcal{T}$, rewrite using 
\eqref{eq:torsion_class} 
\[
a_{ki}^{\gamma_{ki}}=a_{ki}^{\delta_{ki}}\cdot (a_{ki}^{e_{ij}})^{s_i}=_{G} a_{ki}^{\delta_{ki}} \cdot v_i,
\]
where $0\le \delta_{ki}< e_{ki}$ and $v_{i}$ consists of $s_{i}$ copies of 
the right-hand side of \eqref{eq:torsion_class}.
Note that 
\[
|v_i|\le C_0\cdot |s_i|\le C_{0}\cdot|\gamma_{ki}|\leq C_0 C_{k}\cdot|w|^k.
\]
For $ki\notin\mathcal{T}$, put $\delta_{ki}=\gamma_{ki}$ and $v_i=1$. Thus the resulting word
\[
\bar w_k= a_{k1}^{\delta_{k1}}v_1\cdots a_{km_{k}}^{\delta_{km_k}}v_{m_k}
\]
has length at most 
\[
N+m_{k}C_{0}C_{k}\cdot|w|^{k}\le C_{k}\cdot|w|^{k}+m_{k}C_{0}C_{k}\cdot|w|^{k}=\bar{C_k}\cdot|w|^k. 
\]
Repeating the collection procedure for the word $\bar w_k$, as above, we obtain that
\[
\bar w_k=_{G} a_{k1}^{\delta_{k1}}\cdots a_{km_{k}}^{\delta_{km_k}} u_{k+1},
\]
where, using \eqref{eq:pascal_triangle} repeatedly as before,
\begin{eqnarray*}
|u_{k+1}|_{k+\ell}&\le&
C_0^{c/k}\sum_{\iota\le\frac{\ell}{k}}\genfrac(){0pt}{0}{N}{\iota}|\bar w_k|_{k+\ell-\iota k}\\
&\le& C_0^{c/k}\sum_{\iota\le\frac{\ell}{k}}(C_{k}|w|^k)^{\frac{\ell}{k}}|\bar w_k|\\
&\le&
C_0^{c/k}\sum_{\iota\le\frac{\ell}{k}}C_{k}^{\frac{\ell}{k}}
|w|^\ell\bar C_k|w|^k\\
&\le&
\widehat{C}'_{k+1}|w|^{k+\ell},
\end{eqnarray*}
for all $k+1\leq k+\ell\leq c$. 
Combining this with with~\eqref{eq:wk_before_torsion}, we get
\[
w_{k}=_{G} a_{k1}^{\delta_{k1}}\cdots a_{km_k}^{\delta_{km_k}}u_{k+1}w^{(N)}_{k+1}.
\]
Thus setting $w_{k+1}=u_{k+1}w^{(N)}_{k+1}$ and $C_{k+1}=\widehat{C}_{k+1}+\widehat{C}'_{k+1}$ delivers (C), completing the 
inductive step and the proof of the claim.

To prove the theorem it is only left to notice that the claim with $k=c$ suffices, which gives $\kappa=C_{c+1}$.
\end{proof}

\begin{remark}\label{Rem:LowerCentral}Observe that the torsion part of the above argument (obtaining (B)) can be adjusted to allow group relations
\[
a^{e_{ik}}=a_{i, k+1}^{\mu_{i, k+1}}a_{i, k+2}^{\mu_{i, k+2}}\cdots a_{c m_c}^{\mu_{c m_c}}\quad \mbox{for } a=a_{ik}^{\pm 1}
\]
in place of~\eqref{eq:torsion_class}, by processing $a_{k1}$, $a_{k2}$, \ldots, $a_{km_k}$ in succession (in that order). Further, note that though we use the lower central series in Theorem \ref{Lem:PolyCoords}, we only 
need the property
\begin{equation}\label{Eqn:CommutatorGamma}
[\Gamma_{i}, \Gamma_{j}]\leq \Gamma_{i+j}.
\end{equation}
Therefore, any central series having this property will suffice, with 
$c$ being replaced by the length of the series.
\end{remark}
In light of the importance of Theorem \ref{Lem:PolyCoords} to our work, we will 
refer to any Mal'cev basis associated with the lower central series of $G$ 
as a \emph{\malcevgen basis}.  Our algorithmic results 
usually assume that a lower central Mal'cev basis of $G$ is given, though 
by Remark \ref{Rem:LowerCentral} one may substitute a central series satisfying 
(\ref{Eqn:CommutatorGamma}).
If one instead uses a polycyclic basis derived from a cyclic central series (called simply a \emph{Mal'cev basis} in the literature) a similar polynomial bound, albeit of a higher degree, takes place (recall that the functions $p_i$ and $q_i$ are described in Lemma~\ref{le:malcev}, and $m$ is the polycyclic length of $G$).

\begin{lemma}\label{Lem:general_poly_coords}
Let $G$ be a finitely generated nilpotent group with Mal'cev basis $A=\{a_{1},\ldots,a_{m}\}$.
There are a constants $\kappa'$ and $\delta$ depending on $p_i$, $q_i$, and $m$ such that for every word $w$ over $A^{\pm}$,
\begin{equation*}
|\coordsj{w}{i}| \leq \kappa' \cdot |w|^{\delta}
\end{equation*}
for all $i=1,\ldots, c$.
\end{lemma}
\begin{proof}
Assume the lemma holds for all nilpotent groups having a cyclic series of length $m-1$, in particular for $G_{2}=\langle a_{2},\ldots,a_{m}\rangle$. 
Write $w$ in the form 
\begin{equation}
w = w_{1}a_{1}^{\eta_{1}}w_{2}a_{1}^{\eta_{2}}\cdots w_{r}a_{1}^{\eta_{r}}w_{r+1},
\end{equation}
where each $w_{j}$ is a word (possibly empty) in letters $a_{2}^{\pm 1}, \ldots, a_{m}^{\pm 1}$ and $\eta_{j}\in \integers\setminus\{0\}$ for all $j$. The proof then proceeds via a right-to-left collection process utilizing commutation relations to obtain a word of the form \[
w = a_{1}^{\eta} w_{1}'w_{2}'\cdots w_{r}'w_{r+1},
\]
where $\eta=\sum_{i=1}^{r}\eta_{r}$ and $w_{1}'w_{2}'\ldots w_{r}'w_{r+1}$ is an element of $G_{2}$.
\end{proof}
We use Theorem~\ref{Lem:PolyCoords} over the above statement in subsequent arguments, so we leave the details of the proof to the reader.

We would like to mention that our methods, generally speaking, do not extend to polycyclic groups. Indeed, while every polycyclic group has a set of polycyclic generators, below we show that no polynomial bound for coordinates similar to~\eqref{Eqn:PolynomialCoordinateBound} can be met unless the group is virtually nilpotent.

\begin{proposition} Let $H$ be a polycyclic group with polycyclic generators $A=\{a_1,\ldots,a_m\}$. Suppose there is a polynomial $P(n)$ such that if $w$ is a word over $A^{\pm 1}$ of length $n$ then
\begin{equation*}
|\coordsj{w}{i}| \leq P(n)
\end{equation*}
for all $i=1,2,\ldots,m$. Then $H$ is virtually nilpotent.
\end{proposition}
\begin{proof}
Without loss of generality, assume that $P$ is monotone. Let $B_n$ be the ball of radius $n$ centered at $1$ in the Cayley graph of $H$ relative to $A$. Then for every $w\in B_n$, $|\coordsj{w}{i}|$ is bounded by $P(n)$, so $|B_n|\le (2P(n)+1)^m$, i.e., $H$ has polynomial growth. By a result of Gromov~\cite{Gromov_pgrowth:1981}, it follows that $H$ is virtually nilpotent.
\end{proof}

\subsection{Computation of normal forms and the word problem}

Since our algorithms will accept words of $A^{\pm}$ as input, but perform subsequent 
computations using Mal'cev 
coordinates, a necessary first step is to compute these coordinates. The following statement improves on a result in~\cite{EEO}, which proved that  $\mathrm{UT}_{n}(\integers)$ has normal forms computable in logspace.

\begin{theorem}\label{th:compute_malcev}
Let $G$ be a finitely generated nilpotent group with Mal'cev basis $A$.  
There is an algorithm that, given a word $w$ of length $L$ over $A^{\pm}$,
outputs the binary coordinate tuple of
$w$ in space $O(\log L)$ and, simultaneously, time 
$O(L \log^{2} L)$.
\end{theorem}

\begin{proof}
Denote $w=x_1 x_2 \cdots x_L$. 
We hold in memory an array 
$\gamma = (\gamma_{1}, \ldots, \gamma_{m})$, initialized to $(0,\ldots, 0)$, 
which at the end of the algorithm will hold the Mal'cev coordinates of $w$. 
First, we use the functions $\{p_{i}\}_{i}$ to compute the 
coordinates of the product $x_{1}x_{2}$, storing the result in $\gamma$. We then compute, again using the $\{p_{i}\}_{i}$, 
the coordinates of the product $(x_{1}x_{2})\cdot x_{3}$, 
using the saved coordinates of $x_{1}x_{2}$. We continue in this way, performing $m(L-1)$
total evaluations of functions from $\{p_{i}\}_{i}$, and obtain the coordinates of $w$.

Each subword $x_{1}\cdots x_{j}$ has length bounded by $L$, so by Theorem~\ref{Lem:PolyCoords} each of its coordinates may be stored as 
an $O(\log L)$-bit number 
and hence $\gamma$ may be stored in logarithmic space. 
Each evaluation of a function $p_{i}$ involves addition, multiplication, and division (needed to evaluate the functions $r_{k}$ and $s_{k}$)
with $O(\log L)$-bit inputs. The time complexity of these operations is (sub)quadratic 
in the number of bits, hence the overall time complexity is $O(L\log^{2} L)$. 
The standard `schoolbook' arithmetic operations use space linear in the number of 
bits, hence space $O(\log L)$ in this case.
\end{proof}

\begin{remark}\label{re:unary_output}\begin{enumerate}[(a)]
    \item The time complexity is in fact $O(L\cdot f(\log L))$, where $f(k)$ is the complexity  of multiplying two $k$-bit numbers.
Several bounds that are tighter than the $k^{2}$ bound obtained from `long multiplication' are known.
\item \label{li:unary}By employing a binary counting variable, it is straightforward to output the Mal'cev normal form of $w$ as a group word in space $O(\log L)$. However, in that case the algorithm will run in time $O(\max\{L\log^2 L, L^c\})$, where $c$ is the nilpotency class of $G$, because of the size of the output.
\end{enumerate}
\end{remark}

Theorem \ref{th:compute_malcev} of course implies that the word problem 
in $G$ is decidable in logspace.  This was previously known via embedding 
into $\mathrm{UT}_{n}(\integers)$, since \cite{Lipton_Zalc} proved that linear 
groups have \logspace{} word problem.
\begin{corollary}
Every finitely generated nilpotent group has word problem decidable in \logspace{}.
\end{corollary}

Note that Theorem \ref{th:compute_malcev} does not give a solution to the search 
version of the word problem, that is, the problem of writing a trivial element $w$ as a product 
of conjugates of relators.  We give a polynomial-time solution to 
this problem in Theorem~\ref{Thm:WordSearch}. 

While the above `letter-by-letter' application of the functions $p_{i}$ is efficient 
for the initial conversion of input words into coordinates, subsequent 
coordinate computations generally involve a product of a constant number of factors, the coordinates each of which are known, hence it is more efficient to apply the 
polynomials to the coordinates of the factors.

\begin{lemma}\label{Lem:RecomputeMalcev}
Fix a positive integer $k$.  Given integer vectors $v_{1},\ldots, v_{k}$ representing the coordinates of group elements 
$g_{1},\ldots, g_{k}$ and integers $c_{1},\ldots, c_{k}$ one may compute 
\[
\coords{g_{1}^{c_{1}} g_{2}^{c_{2}}\cdots g_{k}^{c_{k}}}
\]
in time $O(L^{2})$ and space $O(L)$, 
where $L$ is the maximum number of bits required to represent any of the $v_{i}$ or any of the $c_{i}$, for $i=1,\ldots,k$.
\end{lemma}
\begin{proof}
Since $k$ is fixed, this follows immediately from the fact that arithmetic operations 
are computable in quadratic time and linear space.
\end{proof}

\subsection{Compressed-word problems and binary Mal'cev coordinates}
Theorem \ref{Lem:PolyCoords} implies that there is a constant $b$, depending on $G$, such that every coordinate of a word of 
length $L$ may be expressed as a $b\log(L)$-bit number.  Therefore every element $g\in G$ 
which may be, on the one hand, represented by a word of length $L$, may be more efficiently represented as an $m$-tuple of $b\log(L)$-bit 
numbers. In this sense, the specification of a Mal'cev basis provides a natural compression scheme on $G$.

In formulating algorithmic problems over $G$ it is therefore natural to expect that input 
group elements are encoded in this scheme, similar to the expectation 
that elements of $\integers$ are encoded as binary numbers rather than in `unary encoding' as words.
On the other hand, straight-line programs provide a general method to formulate algorithmic problems 
over groups with compressed input, and in the case of nilpotent groups the use of 
straight-line programs eliminates the 
need to specify a particular Mal'cev basis.  

The two schemes are in fact 
equivalent for nilpotent groups: a coordinate tuple $(\alpha_{1},\ldots,\alpha_{m})$ encoded using $L$ bits 
is easily converted to a straight-line program of size $O(L)$ producing the normal form $a_{1}^{\alpha_{1}}\cdots a_{m}^{\alpha_{m}}$, and 
we show below that conversion is also efficient in the opposite direction.

\begin{theorem}\label{Thm:CompressedCoords}
Let $G$ be a finitely generated nilpotent group with  
Mal'cev basis set $A$.
There is an algorithm that, given a straight-line program 
$\mathds{A}$ over $A^{\pm}$, computes in time $O(L^{3})$ the coordinate tuple $\coords{\eval{\mathds{A}}}$, where $L=|\mathds{A}|$.  Each 
coordinate of $\eval{\mathds{A}}$ in the output is expressed as an $O(L)$-bit number.

\end{theorem}
\begin{proof}
We may assume that $A$ is a lower central Mal'cev basis. 
Let $B$ be any non-terminal of $\mathds{A}$. 
Since $|\eval{B}|< 2^{L}$, Theorem~\ref{Lem:PolyCoords} gives the bound  
\begin{equation}\label{Eqn:CompressedBound}
|\coordsj{\eval{B}}{i}| \leq \kappa 2^{Lc}
\end{equation}
for each $i=1,\ldots,m$ and so $\coordsj{\eval{B}}{i}$ may be expressed as a $O(L)$-bit number. For every $i$, we compute  
$\coordsj{\eval{B}}{i}$ by induction.  
If the production for $B$ has the form $B\rightarrow x$, 
we simply report 0 or $\pm 1$, as 
the case may be.  

Otherwise, we have $B\rightarrow CD$ and we assume that $\coordsj{\eval{C}}{j}$ and $\coordsj{\eval{D}}{j}$ 
have been computed for $j=1,\ldots, m$, and each coordinate is an $O(L)$-bit number. Then we compute  
$\coordsj{\eval{B}}{i}$ by evaluating $p_{i}$ at $\coords{\eval{C}}$ and $\coords{\eval{D}}$.  Since the inputs 
are $O(L)$-bit numbers, evaluation of $p_{i}$ may be done in time (sub)quadratic in $L$.  
Repeating for each of the $L$ non-terminals of $A$ gives the $O(L^{3})$ bound.
\end{proof}

We can now approach compressed-word versions of various algorithmic problems in $G$ by converting straight-line programs to 
Mal'cev coordinates then 
applying algorithms which work with coordinates.  The first of these is the compressed word problem.  A polynomial-time solution to this problem was also observed 
in~\cite{HLM12}, via reduction to $\mathrm{UT}_{n}(\integers)$, and a new result 
in~\cite{KonikLohrey} shows, using the same reduction, that the compressed word problem in any non-trivial finitely generated torsion-free nilpotent group is complete for the logspace counting class $\mathbf C_=\mathbf L$.
\begin{corollary}
The compressed word problem in every finitely generated nilpotent group is decidable in (sub)cubic time.
\end{corollary}

Throughout this paper we describe several algorithms which take as input 
one or more words over a finitely presented nilpotent group $G$. Each such algorithm 
also comes in two `compressed' versions.

In the `compressed-word' version, all inputs and outputs are straight-line 
programs and the size $L$ of the input is the sum of the sizes of all input 
programs. In the `binary Mal'cev coordinate' version, 
$G$ is provided with a fixed Mal'cev basis $A$ and 
all input and output words are coordinate tuples (relative to $A$)
with each entry written as a binary number. The size $L$ of the input is the total 
number of bits in the input coordinates.

In all cases, the compressed-word 
version works by first computing the Mal'cev coordinates of each 
input straight-line program using Theorem \ref{Thm:CompressedCoords} and then 
invoking the binary Mal'cev coordinate version.  

\section{Matrix reduction, membership problem, and subgroup presentations}
\label{Section:MatrixReduction}

Several algorithmic problems, including the membership problem, 
may be solved by constructing 
an integer matrix from coordinate tuples corresponding to the generating set of a 
subgroup and performing a process similar to Gaussian elimination over $\integers$ to reduce the matrix to a unique `standard form'. 
This approach was detailed in \cite{Sim94}, but without a 
computational complexity analysis.  We review this reduction process and 
analyze its complexity in \S\ref{se:matrix_reduction}, apply it to solve the membership
problem in \S\ref{se:membership_problem}, and use it to compute presentations for 
subgroups in \S\ref{se:presentation}.  It is also essential for computing 
kernels of homomorphisms and thereby solving the conjugacy problem in 
\S\ref{Section:Homomorphisms}. 

\subsection{Matrix reduction}\label{se:matrix_reduction}

Let $h_{1},\ldots, h_{n}$ be elements of $G$ given in normal form by 
$h_{i} = a_{1}^{\alpha_{i1}}\cdots a_{m}^{\alpha_{im}}$, 
for $i=1,\ldots,n$, and let $H=\langle h_{1},\ldots,h_{n}\rangle$.
To the tuple $(h_{1},\ldots,h_{n})$ we associate the \emph{matrix of coordinates}
\begin{equation}\label{Eqn:CoordinateMatrix}
A=
\left(\begin{array}{ccc}
\alpha_{11} &\cdots & \alpha_{1m} \\
\vdots&\ddots & \vdots\\
\alpha_{n1} & \cdots & \alpha_{nm}
\end{array}\right),
\end{equation}
and conversely, to any $n\times m$ integer matrix, we associate an $n$-tuple of elements of $G$, whose Mal'cev coordinates are given as the rows of the matrix, and the subgroup $H$ generated 
by the tuple.
For each $i=1,\ldots,n$ where row $i$ is non-zero, let $\pivot{i}$ be the column of the first non-zero entry (`pivot') in row $i$.
The sequence $(h_{1},\ldots,h_{n})$ is said to be in \emph{standard form} if the matrix of coordinates $A$ is in row-echelon form with no zero rows and its pivot columns are maximally reduced, i.e. if $A$ satisfies the following properties:
\begin{enumerate}[(i)]
\item all rows of $A$ are non-zero (i.e. no $h_{i}$ is trivial),\label{li:std_nontrivial}
\item $\pivot{1} < \pivot{2} < \ldots < \pivot{s}$ (where $s$ is the number of pivots),\label{li:std_echelon}
\item $\alpha_{i \pivot{i}}>0$, for all $i=1,\ldots,n$,\label{li:std_positive}
\item $0\leq \alpha_{k \pivot{i}} < \alpha_{i \pivot{i}}$, for all $1\leq k < i\leq s$, \label{li:std_reduced}
\item  if $\pivot{i}\in\mathcal{T}$, then $\alpha_{i \pivot{i}}$ divides $e_{\pivot{i}}$, for $i=1,\ldots,s$.\label{li:std_torsion}
\end{enumerate}
The sequence is called {\em full} if in addition 
\begin{enumerate}[(i)]\setcounter{enumi}{5}
\item \label{li:std_full}  
$H\cap \langle a_i,a_{i+1},\ldots, a_m\rangle$ is generated by $\{h_{j}\mid \pi_j\ge i\}$, 
for all $1\le i\le m$.
\end{enumerate}

In (\ref{li:std_full}), note that $\{h_{j}\mid \pi_j\ge i\}$ consists of those elements 
having 0 in their first $i-1$ coordinates.  
Let us remark here that (\ref{li:std_full}) 
holds for a given $i$ if and only if the following two properties hold.
\begin{enumerate}[({vi}.i)]
\item For all $1\leq k<j\leq s$ with $\pivot{k}<i$, 
$h_{j}^{h_{k}}$ and $h_{j}^{h_{k}^{-1}}$ are 
elements of $\langle h_{l}\,|\, l>k\rangle$. 
\item For all $1\leq k\leq s$ with $\pivot{k}<i$ and $\pivot{k}\in\mathcal{T}$, 
$h_{k}^{e_{\pivot{k}}/\alpha_{k\pivot{k}}}\in \langle h_{l}\,|\, l>k\rangle$.
\end{enumerate}
Indeed, Lemma \ref{le:malcev}(\ref{le:ab_operation}) 
implies that $h_{j}^{h_{k}}$, $h_{j}^{h_{k}^{-1}}$, and $h_{j}^{e_{\pi_j}/\alpha_{j\pivot{j}}}$ have coordinates 1 through $\pivot{j}$ equal to 0, so the 
forward implication is clear.  Conversely, given an element of 
$H\cap \langle a_i,a_{i+1},\ldots, a_m\rangle$, written as a product of generators of $H$, 
one may first reduce exponents of $h_{1}$ using (vi.ii) 
(if $\pivot{1}\in\mathcal{T}$), and then, 
observing that the exponent sum of $h_{1}$ must be 0, eliminate all occurrences of 
$h_{1}^{\pm 1}$ in conjugate pairs using (vi.i).  Repeating with 
$h_{2},h_{3},\ldots, h_{k}$ where $\pivot{k}<i\leq\pivot{k+1}$, we obtain a word 
in generators $\{h_{j}\,|\,\pivot{j}\geq i\}$.

The importance of full sequences is described in the lemma below, which can be 
found in \cite{Sim94}, Propositions 9.5.2 and 9.5.3.
\begin{lemma}
\label{Lem:UniqueStandardForm}
Let $H\leq G$.  There is a unique full sequence $U=(h_{1},\ldots,h_{s})$ that generates 
$H$.  Further, 
\[
H=\{h_{1}^{\beta_{1}}\cdots h_{s}^{\beta_{s}} \,|\, \beta_{i}\in\integers \mbox{ and $0\leq \beta_{i}<e_{\pivot{i}}$ if $\pivot{i}\in\mathcal{T}$}\}
\]
and $s\leq m$.
\end{lemma}



We define three operations on tuples $(h_{1},\ldots,h_{n})$ of elements of $G$, 
and the corresponding operations on the associated matrix, with the goal of converting $(h_{1},\ldots,h_{n})$ to the unique full-form sequence for  
$H=\langle h_{1},\ldots,h_{n}\rangle$.

\begin{enumerate}[(1)]
\item Swap $h_{i}$ with $h_{j}$.  This corresponds to swapping row $i$ with row $j$.\label{row:swap}
\item Replace $h_i$ by $h_ih_j^l$ ($i\neq j,\; l\in\integers$). This corresponds to replacing row $i$ by $\coords{h_ih_j^l}$.\label{row:substraction}
\item Add or remove a trivial element from the tuple. This corresponds to adding or removing a row of zeros; or (\ref{row:trivial}$'$) a row of the form $(0\;\ldots\; 0\; e_i\; \alpha_{i+1}\;\ldots\; \alpha_m)$, where $i\in \mathcal T$ and $a_i^{-e_i}=a_{i+1}^{\alpha_{i+1}}\cdots a_m^{\alpha_m}$.\label{row:trivial}
\end{enumerate}
Clearly, all three of these operations preserve $H$. By combining these operations, we 
may also  
\begin{enumerate}[(1)]\setcounter{enumi}{3}
\item replace $h_i$ with $h_i^{-1}$, and \label{row:inverse}
\item append to the tuple an arbitrary product $h_{i_1}\cdots h_{i_k}$ of elements in the tuple.\label{row:add_linear}
\end{enumerate}
We also note that these operations, with exception of (\ref{row:trivial}$'$), preserve the property that all rows are Mal'cev coordinate tuples. Further, operation (\ref{row:trivial}$'$) is applied only in Step 3 (see below on p.~\pageref{li:step3}), by completion of which that property is regained.

Using the row operations defined above, we show how to reduce any coordinate matrix to its unique full form, thus 
producing the unique full generating sequence for the corresponding subgroup $H$.
While it is not difficult to see that such reduction is possible, the details of the procedure are essential for our complexity 
estimates.  We make use of the following algorithmic fact regarding greatest common divisors.

\begin{lemma}\label{Lem:GCD}
There is an algorithm that, given integers $a_{1},\ldots,a_{n}$ as binary numbers, computes in time $O(L^{3})$ an expression 
\[
x_{1} a_{1} + \ldots + x_{n} a_{n} = d = \gcd(a_{1},\ldots,a_{n})
\]
with $|x_{i}| \leq \frac{1}{2}\max\{|a_{1}|,\ldots, |a_{n}|\}$, where $L$ is the total number of bits in the input.  If $n$ is fixed, the algorithm may be run in space 
$O(L)$.
\end{lemma}
\begin{proof}
We compute the expression using the binary tree method described in \cite{MH94} Thm. 9.  This computation proceeds in two 
phases. In the first or `bottom-top' phase, we place the integers $a_{1},\ldots,a_{n}$ as the leaves of a binary tree, 
and we compute GCDs of the pairs $(a_{1},a_{2}),(a_{3},a_{4}),\ldots,(a_{n-1},a_{n})$, recording each GCD as its expression as a linear combination of $a_{i}$ and $a_{i+1}$ in the parent node.  We then continue up the tree computing the GCDs of the pairs of parents in the same fashion, 
obtaining $d$ at the root.
This involves invoking the extended Euclidean algorithm (or a more efficient algorithm) at most $n-1$ times, each time with inputs bounded 
by $M=\max\{|a_{1}|,\ldots, |a_{n}|\}$.  Each invocation runs in time $O(\log^{2} M)$, hence the entire phase runs in time $O(L^{3})$.

In the second or `top-bottom' phase, we compute the coefficients $x_{1},\ldots,x_{n}$ (satisfying the given bound) 
from the top of the tree downward, 
using `small' coefficients at each step. Each computation uses a fixed number of arithmetic operations, hence this phase 
also runs in time $O(L^{3})$.

For the space complexity, simply observe that when $n$ is fixed the tree has constant 
size and we use logspace arithmetic operations, which must run in time polynomial in 
$\log L$.
\end{proof}

Let $A_{0}$ be a matrix of coordinates, as in (\ref{Eqn:CoordinateMatrix}) above.   We produce matrices $A_{1}, \ldots, A_{s}$, with $s$ the number of pivots in the full form of $A_{0}$, such that for every $k=1,\ldots, s$ the first 
$\pi_k$ columns of $A_{k}$ form a matrix satisfying (\ref{li:std_echelon})-(\ref{li:std_torsion}), the condition (\ref{li:std_full}) is satisfied for all $i<\pi_{k+1}$, and 
$A_{s}$ is the full form of $A_{0}$.  Here we formally denote $\pivot{s+1}=m+1$.

Set $\pivot{0}=0$ and assume that $A_{k-1}$ has been constructed for some $k\geq 1$. In the steps below we construct $A_{k}$.
We let $n$ and $m$ denote the number of rows and columns, respectively, of $A_{k-1}$.   At all times during the computation, $h_{i}$ denotes 
the group element corresponding to row $i$ of $A_{k}$ and $\alpha_{ij}$ denotes 
the $(i,j)$-entry of $A_{k}$, which is $\coordsj{h_{i}}{j}$.  These may change 
after every operation.

\begin{description}
\item [\sc Step 1]
Locate the column  $\pivot{k}$ of the next pivot, which is the minimum integer $\pivot{k-1}<\pivot{k}\leq m$ such that 
$\alpha_{i \pivot{k}}\neq 0$ for at least one $k \leq i \leq n$. If no such integer exists, then $k-1=s$ and 
$A_{s}$ is already constructed.  Otherwise, set $A_{k}$ to be a copy of $A_{k-1}$ and denote $\pivot{}=\pivot{k}$. Compute a linear expression of $d = {\rm gcd}(\alpha_{k\pivot{}}, \ldots, \alpha_{n\pivot{}})$,
\[
d = l_{k}\alpha_{k \pivot{}} + \cdots +  l_{n}\alpha_{n\pivot{}}.
\]
The coefficients $l_{k},\ldots,l_{n}$ must be chosen so that $|l_{i}|\leq M$ for 
all $i$, 
where $M=\max\{|\alpha_{k\pivot{}}|, \ldots, |\alpha_{n\pivot{}}|\}$.
Let $h_{n+1} = h_{k}^{l_{k}} \cdots h_{n}^{l_{n}}$ and note that $h_{n+1}$ has coordinates of the form
\[
\coords{h_{n+1}} = (0,\ldots,0,d, \ldots)
\]
with $d$ occurring in position $\pivot{}$.
Perform operation~(\ref{row:add_linear}) to append $h_{n+1}$ as row $n+1$ of $A_{k}$.
\item [\sc Step 2] For each $i=k, \ldots, n$, perform row operation~(\ref{row:substraction}) to replace row $i$ by 
\[
\coords{h_i\cdot h_{n+1}^{-\alpha_{i\pivot{}}/d}}. 
\]
For each $i=1,\ldots,k-1$, use (\ref{row:substraction}) to replace row $i$ by 
\[
\coords{h_{i} \cdot h_{n+1}^{-\lfloor \alpha_{i\pivot{}}/d\rfloor}}.
\]
Using (\ref{row:swap}), swap row $k$ with row $n+1$.  At this point, properties 
(\ref{li:std_echelon})-(\ref{li:std_reduced}) hold on the first $k$ columns of $A_{k}$.
\item [\sc Step 3]\label{li:step3} If $\pivot{}\in\mathcal{T}$, we additionally ensure condition~(\ref{li:std_torsion}) as follows. Perform row operation~(\ref{row:trivial}$'$), with respect to $\pivot{}$, to append a trivial element 
$h_{n+2}$ as row $(0,\ldots,0,e_{\pivot{}},\ldots)$ to $A_{k}$. 
Let $\delta=\gcd(d,e_{\pivot{}})$ and compute the linear expression $\delta=n_{1}d+n_{2}e_{\pivot{}}$, with 
$|n_{1}|,|n_{2}|\leq\max\{d,e_{\pivot{}}\}$. Let 
$h_{n+3}= h_{k}^{n_{1}} h_{n+2}^{n_{2}}$ and append this row to $A_{k}$, as row $n+3$. Note that 
$\coords{h_{n+3}}=(0,\ldots,0,\delta,\ldots)$, with $\delta$ in position $\pivot{}$.  
Replace row $k$ by $\coords{h_{k}\cdot h_{n+3}^{-d/\delta}}$ and row $n+2$ by 
$\coords{h_{n+2}\cdot h_{n+3}^{-e_{\pivot{}}/\delta}}$, 
producing zeros in column $\pivot{}$ in these rows. Swap row $k$ with row $n+3$.  
At this point, 
(\ref{li:std_echelon}), (\ref{li:std_positive}), and (\ref{li:std_torsion}) hold 
(for the first $\pivot{k}$ columns) but (\ref{li:std_reduced}) 
need not, since the pivot entry is now $\delta$ instead of $d$. 
For each $j=1,\ldots,k-1$, replace row $j$ by $\coords{h_{j}\cdot h_{k}^{-\lfloor \alpha_{j\pivot{}}/\delta\rfloor}}$, 
ensuring (\ref{li:std_reduced}).

\item [\sc Step 4] Identify the next pivot $\pivot{k+1}$, setting $\pivot{k+1}=m+1$ if 
$\pivot{k}$ is the last pivot. 
We now ensure condition (\ref{li:std_full}) for $i< \pi_{k+1}$. 
Observe that Steps 1--3 preserve $\langle h_{j}\,|\,\pivot{j}\geq i\rangle$ for all 
$i<\pivot{k}$. Hence (\ref{li:std_full}) holds in $A_{k}$ for $i<\pivot{k}$ since it 
holds in $A_{k-1}$ for the same range.  Now consider $i$ in the range 
$\pivot{k}\le i<\pivot{k+1}$. It suffices to prove (vi.i) for all $j>k$ and (vi.ii) for 
$\pivot{k}$ only.


To obtain (vi.i), we notice that $h_{k}^{-1}h_{j}h_{k},h_{k}h_{j}h_{k}^{-1}\in \langle h_\ell \mid \ell>k\rangle$ if and only if 
$[h_{j},h_{k}^{\pm 1}]\in \langle h_{\ell}\mid \ell>k\rangle$. 
Further, note that the subgroup generated by the set 
\[
S_{j}=\{1,h_j, [h_j,h_k],\ldots, [h_j,h_k,\ldots,h_k]\}, 
\]
where $h_k$ appears $m-\pivot{k}$ times in the last commutator, is closed under commutation with $h_k$ since if $h_{k}$ appears more than $m-\pivot{k}$ times then 
the commutator is trivial. An inductive argument shows that the subgroup $\langle S_{j}\rangle$ coincides with $\langle h_j^{h_k^\ell}\mid 0\le \ell\le m-\pi_k\rangle$. Similar observations can be made for conjugation by $h_k^{-1}$. Therefore, appending via operation (\ref{row:add_linear}) rows $\coords{h_j^{h_k^\ell}}$ for all 
$1\leq |\ell| \leq m-\pivot{k}$ and all $k< j\le n+3$ delivers (vi.i) for all $j>k$. 
Note that (vi.i) remains true for $i<\pivot{k}$. 

To obtain (vi.ii), in the case  $\pi_{k}\in\mathcal{T}$, we add row $\coords{h_{k}^{e_{k}/\alpha_{k\pi_{k}}}}$. Note that this element commutes with $h_k$ and therefore (vi.i) is preserved.



\item [\sc Step 5] Using (\ref{row:trivial}), eliminate all zero rows.  The matrix $A_{k}$ is now constructed.
\end{description}

In applying row operation (\ref{row:substraction}) or (\ref{row:add_linear}), 
the magnitude of the largest entry in the matrix may increase. It is 
essential to observe that during the matrix reduction algorithm 
the growth of this value is bounded by a polynomial of fixed degree (depending on $G$).

\begin{lemma}
\label{Lem:MatrixBound}
Let $g_{1},\ldots,g_{t}\in G$ and let $R$ be the 
full form of the associated matrix of coordinates.  Then every entry $\alpha_{ij}$ of $R$ is bounded by 
\[
|\alpha_{ij}| \leq C \cdot L^{K}, 
\]
where 
$L = |g_1| + \cdots + |g_t|$ is the total length of the given elements, and 
$K=m(8c^{2})^{m}$ and $C$ are constants depending on $G$.
\end{lemma}

\begin{proof}

Denote by $A_{0}$ the $t\times m$ matrix of coordinates associated with $(g_1, \ldots, g_t)$.
Following the matrix reduction algorithm described above, we will bound the entries of 
$A_{k}$ in terms of the entries of $A_{k-1}$ and by induction obtain a bound of the entries of $R=A_s$ in terms of the entries of $A_0$. 

For a given $A_{k-1}$, denote by $n$ the number of rows of $A_{k-1}$ and by $N$ the magnitude of the largest entry, i.e.,  
\[
N = \max \{|\alpha_{ij}| \;|\;1\leq i \leq n, 1\leq j \leq m \}.
\]
Observe that for each $1\leq i\leq n$, the element $h_{i}$ corresponding to row $i$ 
of $A_{k-1}$ has length
$|h_i| = |a_1 ^{\alpha_{i1}} \cdots a_m^{\alpha_{im}}|  \leq mN$.
Now in {\sc Step 1} we append the row $h_{n+1}$, which satisfies
\begin{align}
|h_{n+1}| &= | h_k^{l_k} \cdots h_n^{l_n} | = |l_k| |h_k| + \cdots + |l_n| |h_n|\nonumber \\
&\leq  N \big( |h_k| + \cdots + |h_n| \big) \nonumber\\ 
&\leq  mnN^2. \label{bound:nplus1}
\end{align}

Denote by $\alpha_{ij}'$ the $(i,j)$-entry at the end of {\sc Step~2}. 
Since this number is $\coordsj{h_i h_{n+1}^{-\lfloor\alpha_{i\pi}/d\rfloor}}{j}$, except for 
$i=k$, we have, using 
Theorem~\ref{Lem:PolyCoords}, 
\begin{align}
|\alpha_{ij}'| &\leq 
    \kappa \big| h_ih_{n+1}^{-\lfloor \alpha_{i\pi}/d\rfloor} \big|^{c} = 
    \kappa \left( |h_i| + \left\lfloor\frac{|\alpha_{i\pi}|}{d}\right\rfloor |h_{n+1}| \right)^{c} \nonumber\\
&\leq \kappa \big( mN + N\cdot mnN^2 \big)^{c} 
\leq \kappa\big( 2mnN^3 \big)^{c} \nonumber\\ 
&= \kappa(2m)^{c} \cdot n^{c} N^{3c}. \label{bound:step2}
\end{align}
For $i=k$, the tighter bound (\ref{bound:nplus1}) holds.

Proceeding to {\sc Step 2}, denote $E=\max\{e_{i}\,|\,i\in\mathcal{T}\}$.  
The new rows $h_{n+2}$ and 
$h_{n+3}$ satisfy
\begin{align}
|h_{n+2}| &\leq 2\kappa E^{c} \nonumber\\
|h_{n+3}| &= |h_k^{n_1}h_{n+2}^{n_2}| = |n_1||h_k| + |n_2| |h_{n+2}| \nonumber\\
&\leq (EN)(mnN^{2})+(EN)(2\kappa E^{c}) \nonumber \\
&\leq 2mn\kappa E^{c+1}N^{3}.\label{bound:nplus3}
\end{align}
Let $\alpha_{ij}''$ denote the $(i,j)$-entry of $A_{k-1}$ at the end of {\sc Step 4}, and 
recall that {\sc Step 4} only appends rows to the bottom of the matrix.
For row $n+3$ (row $k$ before swapping) we have, for all $j$,
\begin{align}
|\alpha_{(n+3)j}''|&\leq \kappa|h_k h_{n+3}^{-d/\delta} |^{c} \leq \kappa \left( |h_k| + N |h_{n+3}| \right)^{c} \nonumber\\
&\leq \kappa(mnN^{2}+2mn\kappa E^{c+1}N^{4})^{c}\leq
 \kappa^{c+1}(3mn)^{c}E^{c^2+c}N^{4c}.\label{bound:k}
\end{align}
In row $n+2$ we have
\begin{align}
|\alpha_{(n+2)j}''| &\leq \kappa \left| h_{n+2} h_{n+3}^{-e_{\pi}/\delta} \right|^{c} \leq \kappa \left( |h_{n+2}| + E|h_{n+3}| \right)^{c} \nonumber\\ 
&\leq \kappa(2\kappa E^{c}+2mn\kappa E^{c+2} N^{3})^{c} 
\leq \kappa^{c + 1} (4mn)^{c}E^{c^{2}+2c} N^{3c}.\label{bound:nplus2}
\end{align}
Finally, for rows 1 through $k-1$ notice that each of  
$h_{1},\ldots,h_{k-1}$ has length bounded by $m$ times the bound (\ref{bound:step2}), hence
\begin{align}
|\alpha_{jl}''| &\leq \kappa \left|h_jh_k^{-\lfloor a_{j{\pi}}'/\delta\rfloor}\right|^{c} \leq \kappa( |h_j| + |\alpha_{j{\pi}}'| |h_k| )^{c}\nonumber\\
&\leq \kappa\left(m(2m)^{c}\kappa n^{c}N^{3c}+\kappa(2m)^{c} n^{c}N^{3c}\cdot 2mn\kappa E^{c+1}N^{3}\right)^{c}\nonumber\\
&\leq (6\kappa mnEN)^{4c^{2}+3c}.\label{bound:step3}
\end{align} 
Note that at the conclusion of {\sc Step 3}, bound (\ref{bound:step3}) applies to rows 
1 through $k-1$, bound (\ref{bound:nplus3}) applies to the element $h_{k}$ 
(formerly $h_{n+3}$) in row $k$, 
and the maximum of (\ref{bound:k}) and (\ref{bound:nplus2}) applies to all rows after $k$.


In {\sc Step 4} we append all rows of the type $h_j^{h_k^{l}}$ for $1\leq |l| \leq m-\pi_k$ and $k<j\leq n+3$. The entries of such a row are bounded by 
\begin{align}
|\alpha_{pq}''| &\leq \kappa \left| h_j^{h_k^{l}} \right|^{c} \leq \kappa \big(|h_j| + 2 |l| |h_k| \big)^{c} \nonumber\\
&\leq \kappa\left(m\kappa^{c+1}(4mn)^{c}E^{c^{2}+2c}N^{4c}+ 2mm\cdot 2mn\kappa E^{c+1}N^{3}\right)^{c} \nonumber\\
& \leq C''\cdot (nN)^{4c^{2}},
\end{align}
where $C'=(4\kappa mE)^{c^{3}+2c^{2}+3c+1}$. 
If $\pivot{k}\in \mathcal{T}$, we also append the row $h_k^{e_k/\alpha_{k\pivot{k}}}$, 
and so the entries in the final row $r$ of the matrix satisfy 
\begin{align}
|\alpha_{rq}''| &\leq \kappa |h_k^{e_k/\alpha_{k\pi_k}}|^{c} \nonumber\\
&\leq \kappa (Em\cdot 2mn\kappa E^{c+1}N^{3})^{c}.\nonumber \\
&\leq C'' \cdot (nN)^{4c^{2}}.\label{bound:torsion}
\end{align}
Thus the magnitude of each entry of $A_{k}$ is bounded by $C'\cdot (nN)^{4c^{2}}$, where 
$n$ is the number of rows of $A_{k-1}$ and $N$ bounds the magnitude of the entries of 
$A_{k-1}$.  

Next, notice that {\sc Steps 1-3} add three rows, and {\sc Step 4} adds less than 
$2m(n+3)$ rows.  We may bound the number of rows added by $10m\cdot n$.  Consequently, 
the number of rows of $A_{k}$ is bounded by $(10m)^{k}\cdot t$.

A simple inductive argument now shows that every entry of $R$ is bounded by 
\[
C'' \cdot t^{4c^{2}s}N_{0}^{(4c^{2})^{s}},
\]
where $N_{0}$ is the maximum of the absolute value of entries in $A_{0}$ and $C''$ is a constant depending 
on $m$, $c$, $E$, and $\kappa$.  Now 
\[
N_0 \leq  \max_{i} \{\kappa |g_i|^{c}\} \leq  \kappa L^{c}.
\] 
Moreover, $t\leq L$ and $s\leq m$. Therefore the entries of $R$ are bounded by
\[
C'' L^{4c^{2}m}\kappa^{(4c^{2})^{m}}L^{c(4c^{2})^{m}}=C\cdot L^{c(4c^{2})^{m}+4c^{2}m}.
\]
We simplify the exponent of $L$ by using the bound $K=m(8c^{2})^{m}$.

\end{proof}

Lemma \ref{Lem:MatrixBound} allows us to produce a logspace  version of the 
matrix reduction algorithm.  Note that the matrix reduction algorithm, as presented, 
may use more than logarithmic space since the number of rows $t$ of 
the initial matrix $A_{0}$ may be of the same order as the total input size. 


\begin{theorem}
\label{Lem:Compute_std_form}
Let $G$ be a finitely generated nilpotent group with lower-central Mal'cev basis $A$.
There is an algorithm that, given $h_{1},\ldots,h_{n}\in G$, computes the full form 
of the associated matrix of coordinates (relative to $A$) and hence the 
unique full-form sequence $(g_{1},\ldots,g_{s})$ generating $\langle h_{1},\ldots,h_{n}\rangle$. 
The algorithm runs in space $O(\log L)$, where $L=\sum_{i=1}^{n} |h_{i}|$, 
and in time $O(L \log^{3} L)$, and the total length 
of the elements $g_{1},\ldots,g_{s}$ is bounded by a polynomial function of $L$ 
of degree $m(8c^2)^{m}$.
\end{theorem}

\begin{proof}
\emph{Algorithm description.} 
Since the coordinate matrix for $h_{1},\ldots,h_{n}$ cannot be stored in logarithmic 
space, we adopt a piecewise approach, appending one row at a time and reducing.

Form the coordinate matrix $B_{0}$ of the first $m$ elements, $h_1, \ldots, h_m$ and compute its full form $B_{1}$. Append to $B_{1}$ a row corresponding to the coordinates of $h_{m+1}$ and compute the full form $B_{2}$ of this matrix.  Append $h_{m+2}$ to $B_{2}$ and continue in this way until there are no more rows to append. 
Since the subgroup generated by the rows is preserved under row operations, 
the last matrix thus obtained, $B_{n-m}$, is the full form of the matrix of coordinates of $h_1, \ldots, h_n$. 

\emph{Space and time complexity.} 
We first show that at every step the matrix we work with can be stored in logarithmic space. Since each intermediate matrix $B_{l}$ is in full form, 
Lemma \ref{Lem:UniqueStandardForm} ensures that $B_{l}$ has at most $m$ rows.
Hence the number of rows appended the reduction of $B_{l-1}$ to $B_{l}$ is, as seen in the proof 
of Lemma \ref{Lem:MatrixBound}, bounded by $10m^{2}$.  The size of the working 
matrix is therefore never more than $10m^{2}\times m$ (constant with respect 
to the input).

As for the size of the entries,
Theorem~\ref{Lem:PolyCoords} 
shows that each entry $\alpha_{ij}$ of the matrix $B_{0}$ satisfies  
\[
|\alpha_{ij}| \leq \kappa |h_i|^{c} \leq \kappa L^{c}.
\]
Each entry can be encoded using $O(\log L)$ bits and therefore $B_{0}$ can be 
stored in logarithmic space.

Since the matrix $B_{l}$, $1\leq l\leq n-m$, is 
precisely the (unique) full-form matrix for the sequence $h_{1},\ldots,h_{m+l}$,
Lemma \ref{Lem:MatrixBound} ensures that each entry of $B_{l}$ is bounded 
in magnitude by $C\cdot L^{K}$.  The proof of Lemma \ref{Lem:MatrixBound} shows 
that the bound given there holds during all steps of the reduction algorithm, 
hence during the reduction of $B_{l-1}$ to $B_{l}$ no entry can be greater in 
magnitude than $C(m^{2}CL^{K})^{K}$.
Consequently, all intermediate matrices can be stored in logarithmic space.  

It remains to show that the operations that we use can also be executed in logarithmic space and time $O(L\log^{3} L)$.  Computing the linear 
expression of a GCD is performed only on a bounded 
number of integers (the number being bounded by the number of rows, 
see {\sc Step 1} for the worst case), each encoded with $O(\log L)$ bits. It follows that the procedure for doing so 
described in Lemma 
\ref{Lem:GCD} can be carried out in time $O(\log^{3} L)$ and space 
$O(\log L)$. Computation of coordinates is performed initially for each $h_{i}$ using 
Theorem \ref{th:compute_malcev}.  Subsequent coordinate computations involve finding the coordinates of a product of a bounded number of 
factors raised to powers that are $O(\log L)$-bit integers (no larger than the 
greatest entry in the matrix).  The coordinates of each factor are known, hence the computations are performed in 
time $O(\log^{2} L)$ and space $O(\log L)$ by Lemma \ref{Lem:RecomputeMalcev}. 
The other operations (swapping rows, removing zero rows, locating pivot, etc.) are trivial. 

Finally, for each reduction phase (computing $B_{l}$ from $B_{l-1}$), the number of the above operations (GCD, coordinates, etc.) is bounded.  The number 
of phases is bounded by $n\leq L$, hence the time complexity is $O(L\log^{3} L)$.
\end{proof}

\begin{remark}\label{re:time_complexity}
The factor of $L$ in the time complexity arises most heavily from the fact that the number $n$ of input elements can, in general, only be bounded by $L$.  If $n$ is regarded as a fixed number, the most time-consuming computation is computing the Mal'cev 
coordinates and the overall time complexity is reduced to $O(L\log^{2} L)$. 
This remark also applies to Theorems \ref{th:logspace_membership}, 
\ref{Thm:EffectiveCoherence}, and \ref{Thm:KernelAndPreimage}.  
\end{remark}

\begin{remark}\label{re:binary_output}
The output of the algorithm uses binary numbers to 
encode the exponents which are the Mal'cev coordinates of $g_i$. If one is interested in the expression of $g_i$ as {\em words} in generators of $G$, the time complexity grows according to the time it takes to print out the corresponding words, similarly to Remark~\ref{re:unary_output}(\ref{li:unary}) after Theorem~\ref{th:compute_malcev}. Given the polynomial bound on the total length of $g_i$ provided in the statement, the overall algorithm still runs in polynomial time. 
This remark also applies to all the search problems considered below, that is to Theorems~\ref{th:logspace_membership}, \ref{Thm:EffectiveCoherence}, \ref{Thm:KernelAndPreimage}, \ref{Thm:WordSearch}, \ref{Thm:Centralizer}, and \ref{Thm:CP}.
\end{remark}

In order to solve the compressed-word version of the membership (and other) problems, we require a compressed version of matrix reduction that runs in polynomial time.

\begin{lemma}\label{Lem:compressed_std_form}
The compressed version of Theorem \ref{Lem:Compute_std_form}, in which the input 
elements $h_{1},\ldots,h_{n}$ are given either as straight-line programs or 
binary Mal'cev coordinate tuples, runs in time $O(n L^{3})$ where $L$ is the 
total input size. Each $g_{i}$ is output as a straight-line program or binary 
Mal'cev coordinate tuple, of size polynomial in $L$.
\end{lemma}
\begin{proof}
For the compressed-word version, let $\mathds{A}_{1},\ldots,\mathds{A}_{n}$ 
be the input straight-line programs encoding $h_{1},\ldots,h_{n}$. 
We first compute, by Theorem \ref{Thm:CompressedCoords}, 
the coordinate vectors $\coords{\eval{\mathds{A}_{i}}}$ for $i=1,\ldots,n$.  
This operation occurs in time $O(nL^{3})$.
Since $|\eval{\mathds{A}_{i}}|\leq 2^{L}$, we obtain from Theorem \ref{Lem:PolyCoords} 
that each entry of $B_{0}$ is bounded by $\kappa\cdot(2^{L})^{c}$ and hence 
is encoded using $O(L)$, rather than $O(\log L)$, bits. 
Therefore the reduction process described in Theorem \ref{Lem:Compute_std_form} 
runs in time $O(nL^{3})$.  Note that one may instead include all $n$ 
rows in the initial matrix, as in {\sc Steps} 1-5, instead of the piecewise 
approach of Theorem \ref{Lem:Compute_std_form}, since logspace is 
not an issue here.

Since $|\eval{\mathds{A}_{i}}|\leq 2^{L}$, it follows from Lemma \ref{Lem:MatrixBound} that each coordinate in the full-form sequence for $H$ 
is bounded in magnitude by a polynomial function of $2^{L}$.  Then each element of the sequence may be expressed as a 
compressed word of size polynomial in $L$.  

The binary Mal'cev coordinate version simply omits the initial computation of 
coordinates.
\end{proof}

\subsection{Membership problem} \label{se:membership_problem}
We  can  now  apply  the  matrix  reduction  algorithm  to  solve  the  membership problem  in logspace,  and  its  compressed-word  version  in  quartic  time.   We also solve the membership search problem using only logarithmic space.



\begin{theorem}\label{th:logspace_membership}
Let $G$ be a finitely generated nilpotent group. There is an algorithm that, given elements $h_{1},\ldots,h_{n}\in G$ and $h\in G,$ decides whether or not $h$ is an element of the subgroup $H=\langle h_{1},\ldots, h_{n}\rangle$.  The algorithm runs in space logarithmic in $L=|h|+\sum_{i=1}^{n}|h_{i}|$ and time $O(L\log^{3} L)$.  Moreover, the algorithm additionally returns the following: 
\begin{itemize}
    \item the binary coordinates tuples of $g_1,\ldots,g_s\in G$ s.t. $(g_{1},\ldots,g_{s})$ is a full-form sequence for the subgroup $H$ (relative to a lower central Mal'cev basis), and
    \item if $h\in H$, then also the unique binary tuple $(\gamma_1,\ldots,\gamma_s)$ s.t. $h=g_{1}^{\gamma_{1}}\cdots g_{s}^{\gamma_{s}}$.
\end{itemize}
Furthermore, the word length of $g_{1}^{\gamma_{1}}\cdots g_{s}^{\gamma_{s}}$ is bounded by a degree $2m(8 c^{3})^{m}$ polynomial function of $L$.
\end{theorem}

\begin{proof}
\emph{Algorithm}. 
Compute the full form $B$ of the coordinate matrix corresponding to 
$H$ and the full-form sequence $(g_1, \ldots, g_s)$. 
As before, denote by $\alpha_{ij}$ the $(i,j)$-entry of $B$ and by 
$\pivot{1}, \ldots, \pivot{s}$ its pivots.
 
By Lemma~\ref{Lem:UniqueStandardForm}, any element of $H$ can be written as $g_1^{\gamma_1} \cdots g_s^{\gamma_s}$. We show how to find these exponents. Denote $h^{(1)} = h$ and $\coords{h^{(j)}}=(\beta_1^{(j)}, \ldots, \beta_m^{(j)})$, with $h^{(j)}$ being defined below. For $j=1, \ldots, s$, do the following. If $\beta_{l}^{(j)} \neq 0$ for 
any  $1\leq l < \pi_j$, then $h\notin H$. Otherwise, 
check whether $\alpha_{j{\pi_j}}$ divides $\beta_{\pi_{j}}^{(j)}$. If not, then $h\notin H$. If yes, let 
\[
\gamma_{j} = \beta_{\pi_{j}}^{(j)}/{\alpha_{j\pi_j}} \quad \text{and} \quad h^{(j+1)} = g_j^{-\gamma_j}h^{(j)}. 
\]
If $j<s$, continue to $j+1$.  If $j=s$, then $h=g_{1}^{\gamma_{1}}\cdots g_{s}^{\gamma_{s}}\in H$ if $h^{(s+1)}=1$ and 
$h\notin H$ otherwise.



\emph{Complexity and length bound.} 
We first prove, by induction on $s$, that the length of the output $g_1^{\gamma_1}\cdots g_s^{\gamma_s}$ is bounded by a degree 
\[
\delta(s) = c^{s} + m(8c^{2})^{m}\left(\sum_{i=0}^{s-1} c^{i}\right)
\]
polynomial 
function of $L$.
First observe that each $g_{i}$ has 
length bounded by a degree $m(8c^{2})^{m}$ polynomial function of $L$ by Lemma \ref{Lem:MatrixBound}.   For $s=1$ we have by Theorem~\ref{Lem:PolyCoords}, 
\[
|\gamma_{1}| \leq |\beta_{\pi_{1}}^{(1)}| \leq \kappa  |h|^{c} \leq \kappa L^{c} 
\]
hence $h=g_{1}^{\gamma_{1}}$ has length bounded by a degree $c+m(8c^{2})^{m}$ polynomial function of $L$.

Now assume that $g_{1}^{\gamma_{1}}\cdots g_{s-1}^{\gamma_{s-1}}$ has length bounded by a degree $\delta(s-1)$ polynomial function of $L$. 
Then $|h^{(s)}| = |g_{s-1}^{-\gamma_{s-1}}\cdots g_{1}^{-\gamma_{1}}|+ |h|$ is also bounded by a degree $\delta(s-1)$ polynomial function, so by 
Theorem \ref{Lem:PolyCoords} $|\beta_{\pivot{s}}^{(s)}|$, and therefore 
$|\gamma_{s}|$, is bounded by a polynomial function of degree $c\cdot \delta(s-1)$.  It follows that $|g_{s}^{\gamma_{s}}|$ is bounded by a 
degree 
\[
c\cdot\delta(s-1)+m(8c^{2})^{m} = \delta(s)
\]
polynomial function of $L$, and hence $g_{1}^{\gamma_{1}}\cdot g_{s}^{\gamma_{s}}$ obeys the same degree bound.  The bound stated in the theorem 
may be obtained using $s\leq m$.

Regarding complexity, the initial computation of $(g_{1},\ldots,g_{s})$ is performed within the bounds by Lemma \ref{Lem:Compute_std_form}.  
Since the magnitude of each $\gamma_{j}$ is bounded, for all $j$, by a polynomial function of $L$, the coordinates may be computed and stored using logarithmic space 
and in time $O(\log^{2} L)$ by Lemma \ref{Lem:RecomputeMalcev}. The complexity bound then follows from the fact that the recursion has constant depth $s$.



\end{proof}


%

Though the expression of $h$ in terms of the full-form generators $g_{1},\ldots,g_{s}$ 
of $H$ provides a standardized representation, one may also wish to express $h$ in 
terms of the given generators.


\begin{corollary}\label{cor:poly_membership}
Let $G$ be a finitely generated nilpotent group. There is an algorithm that, given elements $h_{1},\ldots,h_{n}\in G$ and $h\in \langle h_1, \ldots, h_n \rangle$, computes an expression $h=h_{i_1}^{\epsilon_{1}}\cdots h_{i_t}^{\epsilon_{t}}$ where $i_j \in \{1, \ldots, n\}$ and $\epsilon_j = \pm 1$. The algorithm runs in time polynomial in $L= |h|+\sum_{i=1}^{n} |h_i|$ and the output has length bounded by a polynomial function of $L$.  
\end{corollary}

\begin{proof}
We modify the matrix reduction algorithm from Section~\ref{se:matrix_reduction} so as to be able to express each $g_i$ as a product of $h_1, \ldots, h_n$. To this end, along with the matrix $A_k$, we store at every step an array $C_k$ which contains the elements corresponding to the rows of the matrix $A_k$, each written as a product of $h_1, \ldots, h_n$. Thus, $C_{s}$ will be an array containing in each entry, $i$, a product of $h_1, \ldots, h_n$ which is equal to $g_i$.

To obtain the array $C_k$ from the information in step $k-1$, we perform on $C_{k-1}$ the corresponding row operation that was performed on $A_{k-1}$, except we record the result not in terms of Mal'cev coordinates, but directly in terms of the words in the array $C_{k-1}$. 

For $1\leq k \leq s$, denote by $L_k$ the length (as a word over the alphabet $\{h_{1}^{\pm 1},\ldots,h_{n}^{\pm 1}\}$) of the largest entry of $C_k$.
Note that at each step, every row has at most two operations performed on it that may increase its length (one application of (\ref{row:substraction}) in {\sc Step 2} and one in {\sc Step 3}), or is newly-created. 
Using the fact that all exponents involved are bounded in magnitude by 
$C\cdot L^{K}$, where $C$ and $K$ are the constants from Lemma \ref{Lem:MatrixBound}, 
it easily follows that there is a poylnomial function $f(L)$ such that 
\[
L_{k}\leq f(L) L_{k-1}.
\]

Since $L_{0}=1$ and $k\leq s$, it follows that $L_{s}$ is bounded by a polynomial function 
of $L$ and therefore the computation of $C_{s}$ is perfomed in polynomial time.


Finally, we can simply substitute the expression of $g_i$ (for $1\leq i \leq s$) in terms of $h_1, \ldots, h_n$ into the expression $h = g_1^{\gamma_1} \cdots g_s^{\gamma_s}$ to obtain an expression $h= h_{i_1}^{\pm 1}\cdots h_{i_t}^{\pm 1}$ with $i_{1}, \ldots, i_t \in \{1, \ldots, n\}$. Since the length (over the alphabet 
$\{g_{1}^{\pm},\ldots,g_{s}^{\pm}\}$) 
of the expression $g_1^{\gamma_1} \cdots g_s^{\gamma_s}$ 
is bounded by a polynomial function of $L$, producing the 
expression $h= h_{i_1}^{\pm 1}\cdots h_{i_t}^{\pm 1}$ occurs in polynomial time.

Observe that this method will not yield a logspace algorithm, 
since the array $C_{s}$ is too large to be stored in memory. 
\end{proof}

The algorithm used in Theorem~\ref{th:logspace_membership} may be combined with the compressed version of matrix reduction (Lemma \ref{Lem:compressed_std_form}) to give a polynomial-time solution 
to the compressed membership problem.

\begin{theorem}\label{Thm:CompressedMembership}
Let $G$ be a finitely generated nilpotent group. There is an algorithm that, given compressed words $\mathds{A}_{1},\ldots,\mathds{A}_{n}, \mathds{B}$ 
(or binary Mal'cev coordinate tuples) over $G$, decides whether or not $\eval{\mathds{B}}$ belongs to the subgroup generated by $\eval{\mathds{A}_{1}},\ldots,\eval{\mathds{A}_{n}}$. The algorithm runs in time $O(n L^{3})$, where $L=|\mathds{B}|+|\mathds{A}_{1}|+\ldots+|\mathds{A}_{n}|$.  
\end{theorem}

As in Theorem \ref{th:logspace_membership}, the algorithm may also compute the unique expression of $\eval{\mathds{B}}$ in terms of the 
standard-form sequence for $\langle\eval{\mathds{A}_{1}},\ldots,\eval{\mathds{A}_{n}}\rangle$.

\subsection{Subgroup presentations}
\label{se:presentation}
We now apply matrix reduction to show that finitely generated nilpotent groups 
are logspace effectively coherent: given a finitely generated subgroup $H$, 
we can in logspace and quasilinear time compute a consistent nilpotent 
presentation for $H$. By the \emph{size} of a presentation we mean the 
number of generators plus the sum of the lengths of the relators.

\begin{theorem}\label{Thm:EffectiveCoherence}
Let $G$ be a finitely generated nilpotent group.  There is an algorithm that, given $h_{1},\ldots,h_{n}\in G$, computes  
a consistent nilpotent presentation for the subgroup $H=\langle h_{1},\ldots, h_{n}\rangle$.  
The algorithm runs in space 
logarithmic in $L=\sum_{i=1}^{n} |g_{i}|$ and time $O(L\log^{3} L)$, the 
size of the presentation is bounded by a degree $2m(8c^{3})^{m}$ polynomial function of $L$, and binary numbers are used to encode exponents appearing in relators of the 
presentation.
\end{theorem}
\begin{proof}
\emph{Algorithm.} Begin by computing the full sequence $(g_{1},\ldots,g_{s})$ 
for $H$ using Lemma \ref{Lem:Compute_std_form}. Let 
$H_{i}=\langle g_{i},g_{i+1},\ldots,g_{s}\rangle$. We claim that 
\[
H=H_{1} \geq H_{2} \ldots \geq H_{s+1}=1
\]
is a cyclic central series for $H$.  From property (\ref{li:std_full}) we have 
\[
H_{i} = H\cap \langle a_{\pi_{i}},a_{\pi_{i}+1},\ldots,a_{m}\rangle.
\]
Since 
$\langle a_{\pi_{i}},\ldots,a_{m}\rangle$ is a normal subgroup of $G$, it follows 
that the above is a normal series, and since 
$H_{i}/H_{i+1}=\langle g_{i}H_{i+1}\rangle$ the series is cyclic. For $i<j$, 
$[g_{i}, g_{j}]\in \langle a_{\pi_{j}+1},\ldots,a_{m}\rangle\cap H\leq H_{j+1}$ hence 
the series is central.  We conclude that $(g_{1},\ldots,g_{s})$ is a Mal'cev basis 
for $H$, so it suffices to compute the relators 
(\ref{stdpolycyclic1})-(\ref{stdpolycyclic3}) in order to give a consistent 
nilpotent presentation of $H$. 
The order $e_{i}'$ of $g_{i}$ modulo $H_{i+1}$ is simply 
$e_{i}/\coordsj{g_{i}}{\pivot{i}}$. We establish 
each relation (\ref{stdpolycyclic1}) by invoking Theorem \ref{th:logspace_membership} 
with input $g_{i}^{e_{i}'}$ and $H_{i+1}=\langle g_{i+1},\ldots, g_{s}\rangle$.  
Since $g_{i}^{e_{i}'}\in H_{i+1}$ and $(g_{i+1},\ldots,g_{s})$ is the unique full sequence for $H_{i+1}$, the membership algorithm returns the expression 
on the right side of (\ref{stdpolycyclic1}).  Relations (\ref{stdpolycyclic2}) and 
(\ref{stdpolycyclic3}) are established similarly.

\emph{Complexity and size of presentation.} Computation of the full sequence and 
invoking Theorem \ref{th:logspace_membership} occurs within the stated space 
and time bounds.  Regarding the size of the presentation, 
the number of generators is $s\leq m$, the number of relations is bounded 
by $s+2\genfrac(){0pt}{}{s}{2}$, and each relation has length obeying the degree $2m(8c^{3})^{m}$ 
bound of Theorem \ref{th:logspace_membership}.
\end{proof}
\begin{remark}
By Remark~\ref{Rem:LowerCentral}, the Mal'cev basis for $H$ obtained in Theorem \ref{Thm:EffectiveCoherence} may be used in Theorem \ref{Lem:PolyCoords}.  Indeed, one may group the Mal'cev generators 
$g_{1},\ldots,g_{s}$ into the appropriate terms 
of the lower central series of $G$, i.e. set 
\[
H_{i}'=\langle g_{k_{i}},g_{k_{i}+1},\ldots,g_{s}\rangle
\]
where $k_{i}$ is the least index such that 
$g_{k_{i}}\in \Gamma_{i}\setminus\Gamma_{i-1}$.  Then 
$H=H_{1}'\geq H_{2}'\geq\ldots\geq H_{c+1}'=1$ is a central series 
for $H$ in which $[H_{i}',H_{j}']\leq H_{i+j}'$ and $g_{1},\ldots,g_{s}$ is 
an associated Mal'cev basis.
\end{remark}



The compressed-word version of Theorem \ref{Thm:EffectiveCoherence}, running in polynomial time, follows immediately.  However, the relators 
of the presentation are provided as straight-line programs. This is 
unconventional, but we may convert to an `uncompressed' 
presentation, \emph{of polynomial size}, using the following construction.

Consider any presentation $\mathcal{P}$ of a group $H$ in which each relator $R$ is the 
output of a straight-line program $\mathds{A}_{R}$. We apply the following construction 
to $\mathcal{P}$. 
Add 
the set of non-terminals $\mathcal{A}$ of all $\mathds{A}_{R}$ to the generators of $\mathcal{P}$.  For each production rule $A\rightarrow BC$, add the relation 
$A=BC$ to the relations of $\mathcal{P}$; for each production $A\rightarrow x$ add the relation $A=x$; for $A\rightarrow \epsilon$ add the relation $A=1$. 
Replace the original relator $\eval{\mathds{A}_{R}}$ by the root non-terminal of 
$\mathds{A}_{R}$ and denote the resulting presentation $\mathcal{P}'$.

It is clear that the sequence of Tietze transformations which at each step removes the greatest non-terminal and its production rule converts 
$\mathcal{P}'$ back to $\mathcal{P}$. Thus $\mathcal{P}$ and $\mathcal{P}'$ present 
isomorphic groups. As the outputs $\eval{\mathds{A}_{R}}$ are, in many cases, 
exponentially 
longer than the size of the programs $\mathds{A}_{R}$, this construction 
significantly reduces the size of the presentation.

\begin{theorem}\label{Thm:CompressedEffectiveCoherence}
Let $G$ be a finitely presented nilpotent group.  There is an algorithm that, given a finite set $\mathds{A}_{1},\ldots,\mathds{A}_{n}$ of 
straight-line programs over $G$ (or binary Mal'cev coordinate tuples), 
computes a presentation for the subgroup 
$\langle \eval{\mathds{A}_{1}},\ldots, \eval{\mathds{A}_{n}}\rangle$.  
The algorithm runs in time $O(n L^{3})$, where $L=\sum_{i=1}^{n} |\mathds{A}_{i}|$, and the size of the presentation is bounded by a polynomial function of $L$.
\end{theorem}

We can also use this construction to remove the binary numbers from the 
presentation found in Theorem \ref{Thm:EffectiveCoherence} and at the same time 
produce a smaller presentation.  However, the presentation will no longer fit 
the definition of a nilpotent presentation due to the presence of extra relators 
used for compressing the exponents.

\begin{corollary}\label{Cor:NonBinaryCoherence}
The algorithm of Theorem \ref{Thm:EffectiveCoherence} may be modified to produce 
a presentation of $H$ (without binary numbers) of size logarithmic in $L$.  The 
algorithm runs in logarithmic space and time $O(L\log^{3} L)$.
\end{corollary}
\begin{proof}
The number of generators and relators is bounded by a constant, as observed
above.  Each power $a^{\beta}$ 
appearing in a relator is compressed using a straight-line program of size $O(\log\beta)$. 
Since each exponent $\beta$ is bounded by a polynomial function of $L$, the total size 
of the straight-line programs, and thus of the presentation, is logarithmic in $L$.
\end{proof}

\section{Homomorphisms and the conjugacy problem}
\label{Section:Homomorphisms}

Using matrix reduction, \cite{Sim94} showed how to compute the kernel of a homomorphism and compute a preimage for a given element under a homomorphism. 
We prove in \S\ref{se:kernels_presentations} 
that these algorithms may be run in logarithmic space and 
compressed-word versions in polynomial time.  In \S\ref{Section:CP} we 
apply these algorithms to solve the conjugacy problem.

\subsection{Kernels} 
\label{se:kernels_presentations}

For fixed nilpotent groups $G$ and $H$, we may specify a homomorphism 
$\phi$ from a subgroup $K\leq G$ to $H$ via a generating set $(g_{1},\ldots,g_{n})$ 
of $K$ and a list of elements $h_{1},\ldots,h_{n}$ where $\phi(g_{i})=h_{i}$, 
$i=1,\ldots,n$.  For such a homomorphism, we consider the problem of finding a 
generating set for its kernel, and given $h\in \phi(K)$ finding $g\in G$ 
such that $\phi(g)=h$.  Both problems are solved using matrix reduction 
in the group $H\times G$.


\begin{theorem}\label{Thm:KernelAndPreimage}
Fix finitely generated nilpotent groups $G$ and $H$, with lower-central Mal'cev 
basis $A$ and $B$, and let $c$ be the sum 
of their nilpotency classes and $m=|A|+|B|$. 
There is an algorithm that, given 
\begin{itemize}
\item a subgroup $K=\langle g_{1},\ldots,g_{n}\rangle\leq G$,
\item a list of elements $h_{1},\ldots,h_{n}\in H$, such that the mapping $g_i\mapsto h_i$, $1\le i\le n$, extends to a homomorphism $\phi:K\rightarrow H$,
\item optionally, an element $h\in H$ guaranteed to be in the image of $\phi$, 
\end{itemize}
computes binary coordinates tuples (relative to the Mal'cev basis $A$) for
\begin{enumerate}[(i)]
\item a generating set $X$ for the kernel of $\phi$, and \label{kernel}
\item an element $g\in G$ such that $\phi(g)=h$.\label{preimage}
\end{enumerate}
The algorithm runs in space logarithmic in 
$L=|h|+\sum_{i=1}^{n} (|h_{i}|+|g_{i}|)$ and time $O(L \log^{3} L)$, $X$ consists of at most $m$ elements, and 
there is a degree $m(8c^{2})^{m}$ polynomial function of $L$ that bounds the word length of each element of $X$ and a 
degree $2m (8c^{3})^{m}$ polynomial function of $L$ that bounds the word length of $g$.
\end{theorem}
\begin{proof} 
Let $c_{1}$ be the nilpotency class of $G$ and $c_{2}$ that of $H$ and 
consider the nilpotent group $H\times G$.  From the lower central series 
$H=\Gamma_{0}(H)\geq \Gamma_{1}(H)\geq\ldots\geq\Gamma_{c_{2}}(H)=1$ of $H$ and 
$G=\Gamma_{0}(G)\geq \Gamma_{1}(G)\geq\ldots\geq\Gamma_{c_{1}}(G)=1$ of $G$, form 
for $H\times G$ the central series
\begin{equation}\label{Eqn:HGseries}
H\times G=\Delta_{0}\geq \Delta_{1}\geq \ldots\geq \Delta_{c}=1
\end{equation}
where $\Delta_{i}=\Gamma_{i}(H)G$ if $0\leq i< c_{2}$ and 
$\Delta_{i}=\Gamma_{i-c_{2}}(G)$ if $ c_{2}\leq i\leq c=c_{2}+c_{1}$.  Notice 
that this series has the property that 
$[\Delta_{i},\Delta_{j}]\leq \Delta_{i+j}$, since both lower central series have this 
property and the subgroups $H$ and $G$ commute in $H\times G$. Letting 
$B=(b_{1},\ldots,b_{m_{2}})$ and $A=(a_{1},\ldots,a_{m_{1}})$, we 
see that $(b_{1},\ldots,b_{m_{2}},a_{1},\ldots,a_{m_{1}})$ is a 
Mal'cev basis associated with (\ref{Eqn:HGseries}) for which  
Theorem \ref{Lem:PolyCoords} applies.

Let $Q=\langle h_{i} g_{i}\, |\, 1\leq i\leq n\rangle$ and $P=\{ \phi(k)k \;|\; k\in K\}$. We claim that $Q=P$. First, 
observe that $P$ is a subgroup of $H\times G$. Now, each generator $h_{i}g_{i}$ of $Q$ is in $P$, therefore $Q\leq P$. Conversely, if
$\phi(k)k\in P$, then 
$k=g_{i_{1}}^{\epsilon_{1}}\cdots g_{i_{t}}^{\epsilon_{t}}$ hence 
$\phi(k)k= (h_{i_{1}}g_{i_{1}})^{\epsilon_{1}}\cdots (h_{i_{t}}g_{i_{t}})^{\epsilon_{t}}\in Q$.

Let $W=(v_{1} u_{1}, \ldots, v_{s} u_{s})$ be the sequence in full form 
for the subgroup $Q$, 
where $u_{i}\in G$ and $v_{i}\in H$.   Let  $0\leq r\leq s$ 
be the greatest integer such that $v_{r}\neq 1$ (with $r=0$ if all $v_{i}$ are 1).
Set $X=(u_{r+1},\ldots,u_{n})$ and $Y=(v_{1},\ldots,v_{r})$.  We claim that 
$X$ is the full-form sequence for the kernel of $\phi$ and $Y$ is the full-form sequence for the image.

From the fact that $W$ is in full form, it follows that both $X$ and $Y$ are in full form. 
Since $Q=P$, it follows that $v_{i}=\phi(u_{i})$ for $i=1,\ldots,s$. Hence $X$ is contained in the kernel and $Y$ in the image. Now 
consider an arbitrary element 
$\phi(g)g$ of $Q$.  
There exist integers $\beta_{1},\ldots,\beta_{s}$ such that 
\[
\phi(g) g = (v_{1}u_{1})^{\beta_{1}}\cdots (v_{s}u_{s})^{\beta_{s}}=
	v_{1}^{\beta_{1}}\cdots v_{r}^{\beta_{r}} \cdot u_{1}^{\beta_{1}}\cdots u_{s}^{\beta_{s}}.
\]
If $\phi(g)$ is any element of the image, the above shows that $\phi(g)=v_{1}^{\beta_{1}}\cdots v_{r}^{\beta_{r}}$, hence 
$Y$ generates the image.  If $g$ is any element of the kernel, then $1=\phi(g)=v_{1}^{\beta_{1}}\cdots v_{r}^{\beta_{r}}$. 
Since the integers $\beta_{1},\ldots,\beta_{s}$ are unique (taking $0\leq \beta_{i}<e_{i}$ when $e_{i}\in\mathcal{T}$), 
we have $\beta_{1}=\cdots=\beta_{r}=0$.  Consequently 
$g=u_{1}^{\beta_{1}}\cdots u_{s}^{\beta_{s}}=u_{r+1}^{\beta_{r+1}}\cdots u_{s}^{\beta_{s}}$, hence $X$ generates the kernel.

Then to solve (\ref{kernel}), it suffices to compute $W$ using Lemma \ref{Lem:Compute_std_form} and return $X$ (or $\{1\}$ if $r=s$).
To solve 
(\ref{preimage}), use Theorem \ref{th:logspace_membership} to express $h$ as $h=v_{1}^{\beta_{1}}\cdots v_{r}^{\beta_{r}}$ and 
return $g=u_{1}^{\beta_{1}}\cdots u_{r}^{\beta_{r}}$.  
The length and complexity bounds are as given in Lemma \ref{Lem:Compute_std_form} and 
Theorem \ref{th:logspace_membership}.
\end{proof}

\begin{remark}
Notice that if $G,H$ are as in Theorem~\ref{Thm:KernelAndPreimage}, it is possible to check, given lists of elements $g_1,\ldots,g_n\in G$ and $h_{1},\ldots,h_{n}\in H$, whether the mapping $g_i\mapsto h_i$, $1\le i\le n$, extends to a homomorphism $\phi:\langle g_1,\ldots,g_n\rangle =K\rightarrow H$. Indeed, by Theorem~\ref{Thm:EffectiveCoherence} we can compute a finite presentation for $K$, and then for each relator $r$ of that presentation check if the induced value $\phi(r)$ is trivial in $H$ using the solution to the word problem in $H$, by Theorem~\ref{th:compute_malcev}. Note that both of the mentioned theorems run within the time and space bounds of Theorem~\ref{Thm:KernelAndPreimage}.
Therefore, in Theorem~\ref{Thm:KernelAndPreimage} we can omit the requirement that $g_i\mapsto h_i$, $1\le i\le n$, extends to a homomorphism. Instead, we can check if it is the case within the same time and space bounds, and if not, output \emph{``Not a homomorphism''} and a word in $r=g_{i_1}\cdots g_{i_s}$ representing $1$ in $G$ s.t. $h_{i_1}\cdots h_{i_s}$ does not represent $1$ in $H$ (like in Theorem~\ref{Thm:EffectiveCoherence}, binary numbers are used to encode exponents appearing in $r$).

\end{remark}

We can apply Theorem \ref{Thm:KernelAndPreimage} to solve the search version 
of the word problem in logspace and quasilinear time. 

\begin{theorem} \label{Thm:WordSearch}
Let $G$ be a finitely presented nilpotent group.  There is an algorithm that, given 
a word $w$ in generators of $G$ which represents the identity element, produces 
a word $w'$ written as a product of conjugates of relators of $G$ 
such that $w'$ is freely equivalent to $w$. The algorithm runs in space logarithmic in $L=|w|$ and time $O(L\log^{2} L)$, the length of $w'$ is bounded by a degree $2m(8c^{3})^{m}$ polynomial function of $L$, and binary numbers are used to encode exponents appearing in $w'$.
\end{theorem}

\begin{proof}
Let $G = \langle X | R \rangle$.  Denote by $F(X)$ the free nilpotent group on $X$ 
of class $c$, by $N$ the normal closure of $R$ in $F(X)$, and by $\phi$ the canonical epimorphism $F(X) \rightarrow G$. 
Using Theorem~\ref{Thm:KernelAndPreimage}, compute the full-form sequence 
$Y=(y_{1},\ldots,y_{s})$ for $N = \ker \phi$. By exhaustive search, write each $y_{i}$ 
as a product of 
conjugates of elements of $R$.  Note that this step is a precomputation.

Since $w$ represents the trivial element, $w$ is in $N$, so we run the algorithm from 
Theorem~\ref{th:logspace_membership} in order to express $w$ in the form 
$w=y_{1}^{\gamma_{1}}\cdots y_{s}^{\gamma_{s}}$. Then substitute the expression 
of each $y_{i}$ as a product of conjugates of relators of $G$. 

The space complexity, time complexity, and length bound are as given in Theorem~\ref{th:logspace_membership}, noting that $N$ is fixed and hence Remark 
\ref{re:time_complexity} applies giving $O(L\log^{2} L)$ rather than $O(L\log^{3} L)$. 
\end{proof}





In order to solve the compressed conjugacy problem, we need a compressed 
version of Theorem \ref{Thm:KernelAndPreimage}, in which all input 
and output words are encoded as straight-line programs.
We follow the same algorithmic steps, using Lemma \ref{Lem:compressed_std_form} to 
compute $W$ in Mal'cev coordinate form and Theorem \ref{Thm:CompressedMembership} to compute $\beta_{1},\ldots,\beta_{r}$. 

\begin{theorem}\label{Thm:CompressedKernel}
The compressed version of Theorem \ref{Thm:KernelAndPreimage}, in which 
all input words are given as straight-line programs or binary Mal'cev coordinate 
tuples and the output is given in the same form, runs in time $O(n L^{3})$.
\end{theorem}


\subsection{Conjugacy problem and centralizers}\label{Section:CP}

The decidability of the conjugacy problem in finitely generated nilpotent groups has been known since \cite{Rem69} and \cite{For76} proved that 
every polycyclic-by-finite group is \emph{conjugately separable}: if two elements are not conjugate, there exists a finite quotient in which 
they are not conjugate.  While this fact leads to an enumerative solution to the conjugacy problem, a much more practical approach, using 
matrix reduction and homomorphisms, was given in \cite{Sim94}.  

The computational complexity of this algorithm, however, was not analyzed.  We show here that it may be run using logarithmic space, and that 
the compressed-word version runs in polynomial time.  A necessary step in the solution is the computation of centralizers, which is achieved by induction on 
the nilpotency class $c$ of $G$.

\begin{theorem}\label{Thm:Centralizer}
Let $G$ be a finitely generated nilpotent group. 
There is an algorithm that, given $g\in G$, computes a generating set $X$ (in the form of the corresponding binary coordinates tuples) for the centralizer of $g$ in $G$. The algorithm runs in space logarithmic in $L=|g|$ and time $O(L\log^{2} L)$, $X$ contains at most $m$ elements, and there is a degree $(16m(c+1)^{2})^{3(c+1)m}$ polynomial function of $L$ that bounds the length of each element of $X$.
\end{theorem}
\begin{proof}
Let $G=\Gamma_{0}\geq \Gamma_{1}\geq\ldots\geq\Gamma_{c+1}=1$ be the lower central 
series of $G$ (or another series satisfying (\ref{Eqn:CommutatorGamma})).  
We proceed by induction on $c$. If $c=1$, then $G$ is abelian and $C(g)=G$ so we return $a_{11}, a_{12},\ldots,a_{1m_{1}}$.  Assume 
that the theorem holds for finitely presented nilpotent groups of class $c-1$, 
in particular for $G/\Gamma_{c}$. We use the series 
$\{\Gamma_{i}/\Gamma_{c}\}_{i=0}^{c}$ for $G/\Gamma_{c}$.


Compute a generating set $K=\{k_{1}\Gamma_{c},\ldots,k_{m-m_{c}}\Gamma_{c}\}$ for the centralizer of $g\Gamma_{c}$ in $G/\Gamma_{c}$. 
Let $J=\langle k_{1},\ldots,k_{m-m_{c}}, a_{c1},a_{c2},\ldots,a_{cm_{c}}\rangle$, 
which is the preimage of 
$\langle K\rangle$ under the homomorphism 
$G\rightarrow G/\Gamma_{c}$.  
Define $f:J\rightarrow G$ by 
\[
f(u) = [g,u].
\]
Since $u\in J$, $u$ commutes with $g$ modulo $\Gamma_{c}$, hence $[g,u]\in \Gamma_{c}$ and so $\mathrm{Im}(f)\subset \Gamma_{c}$. We claim that $f$ is a 
homomorphism.  Indeed, 
\[
f(g,u_{1}u_{2})=[g,u_{1}u_{2}]=[g,u_{2}][g,u_{1}][[g,u_{1}],u_{2}],
\]  
but $[g,u_{1}]\in \Gamma_{c}$ hence $[[g,u_{1}],u_{2}]\in \Gamma_{c+1}=1$, and 
$[g,u_{1}]$ and $[g,u_{2}]$ commute, both being elements of the abelian group 
$\Gamma_{c}$. 

The centralizer of $g$ is precisely the kernel of $f:J\rightarrow \Gamma_{c}$, since if $h$ commutes with $g$, then $h\Gamma_{c}\in \langle K\rangle$ 
so $h\in J$.  We compute a generating set using 
Theorem \ref{Thm:KernelAndPreimage}.

\emph{Complexity and length bound.} We proceed to prove the length bound by induction on $c$.  For $c=1$, the algorithm returns a single symbol.  
In the inductive case, 
the algorithm is invoked with input $g\Gamma_{c}$, which has size $L$.  Each returned generator $k_{i}$ has length bounded by some number $\kappa$ which is a polynomial function 
of $L$ of degree $(16mc^{2})^{3cm}$.  
Since for each $i=1,\ldots,m-m_{c}$, $[g, k_{i}]$ has length at most $4\kappa$, 
each input-output pair $(k_{i},[g,k_{i}])$ or $(a_{cj},[g,a_{cj}])$ has 
length at most $5\kappa$ hence 
the total input to the kernel algorithm of Theorem \ref{Thm:KernelAndPreimage} has length 
bounded by $5m\kappa$.  The length of the Mal'cev basis of $\Gamma_{c}$ is 
$m_{c}\leq m$, and its nilpotency class is 1, so we use $2m$ and $c+1$ in 
place of $m$ and $c$ in the degree bound in 
Theorem \ref{Thm:KernelAndPreimage}.  Hence the 
length of each output generator has length bounded by a polynomial function of $L$ of degree
\[
(16mc^{2})^{3cm} \cdot 2m(8(c+1)^{2})^{2m} < (16m(c+1)^{2})^{3(c+1)m},
\]
as required.

The logarithmic space bound follows from the fact that in each recursive call, the number $c$ of which is constant, 
the total size of the input is bounded by a polynomial function 
of $L$ and the kernel algorithm runs in logarithmic space.  The time complexity of $O(L\log^{2} L)$ arises entirely from the computation 
of the Mal'cev coordinates of $g\Gamma_{c}$. 
Indeed, 
the total bit-size of the coordinate tuple of $g$ is logarithmic in $L$ and each invocation of the kernel algorithm 
is made with at most $m$ elements of bit-size logarithmic in a polynomial function of $L$, hence logarithmic in $L$, so each of the $c$ recursive 
calls terminates in time $O(\log^{3} L)$ by Theorem \ref{Thm:CompressedKernel}.

\end{proof}

We also require a compressed version of Theorem \ref{Thm:Centralizer}. The 
algorithm is the same, with a time complexity of $O(L^{3})$ arising 
both from the initial computation of Mal'cev coordinates and the fact that 
Theorem \ref{Thm:CompressedKernel} is invoked with at most $m$ elements 
of bit-size $O(L)$.

\begin{theorem}\label{Thm:CompressedCentralizer}
Let $G$ be a finitely generated nilpotent group.  There is an algorithm that, 
given a straight-line program $\mathds{A}$ over $G$ (or a binary Mal'cev coordinate 
tuple), computes in time $O(L^{3})$ a generating set of the centralizer 
of $\eval{\mathds{A}}$ in $G$, where $L=|\mathds{A}|$.  The generators are given 
as straight-line programs (or binary Mal'cev coordinate tuples) of size polynomial 
in $L$.
\end{theorem}

We may now solve the conjugacy and conjugacy search problems in logspace, again 
by induction on the nilpotency class of $G$.

\begin{theorem}\label{Thm:CP}
Let $G$ be a finitely presented nilpotent group. 
There is an algorithm that, given $g,h\in G$, either 
\begin{enumerate}[(i)]
\item produces binary coordinates tuple for $u\in G$ such that $g=u^{-1}hu$, or 
\item determines that no such element $u$ exists.  
\end{enumerate}
The algorithm runs in space logarithmic in $L=|g|+|h|$ and time $O(L\log^{2} L)$, and the word length of $u$ is bounded by a degree  
$2^{m}(6mc^{2})^{m^{2}}$ polynomial function of $L$.
\end{theorem}
\begin{proof} \emph{Algorithm.}  
We proceed by induction on $c$.
If $c=1$, then $G$ is abelian and $g$ is conjugate to $h$ if and only if 
$g=h$. If so, we return $u=1$.

Now assume $c>1$, and that the theorem holds for any nilpotent group of class 
$c-1$, in particular for $G/\Gamma_{c}$.  Apply the algorithm to $g \Gamma_{c}$ and 
$h \Gamma_{c}$, using the series $\{\Gamma_{i}/\Gamma_{c}\}_{i=0}^{c}$ 
for $G/\Gamma_{c}$. 
If these elements are not conjugate, then 
$g$ and $h$ are not conjugate and we return `No'. Otherwise, we obtain $v \Gamma_{c}\in G/\Gamma_{c}$. 
such that $g \Gamma_{c} = h^{v} \Gamma_{c}$. 

Let $\phi:G\rightarrow G/\Gamma_{c}$ be the canonical homomorphism, 
$J=\phi^{-1}(C(g\Gamma_{c}))$, and define $f:J\rightarrow \Gamma_{c}$ by $f(x)=[g,x]$.  As in the proof of Theorem \ref{Thm:Centralizer}, 
the image of $f$ is indeed 
in $\Gamma_{c}$ and $f$ is a homomorphism. We claim that 
$g$ and $h$ are conjugate if and only if $g^{-1}h^{v}\in f(J)$.  Indeed, if there exists $w\in G$ such that $g=h^{vw}$, then 
\[
1 \cdot \Gamma_{c} = g^{-1} w^{-1} h^{v} w \cdot \Gamma_{c} = [g,w] \cdot \Gamma_{c} 
\]
hence $w\in J$, so $w^{-1}\in J$ as well.  Then $g^{-1}h^{v}=[g,w^{-1}] \in f(J)$, as required.  The converse is immediate. 
So it suffices to express, if possible, $g^{-1}h^{v}$ as $[g,w]$ with $w\in J$, in which case the conjugator is $u=vw^{-1}$.
 
Compute a generating set $\{w_{1} \Gamma_{c}, \ldots, w_{m-m_{c}} \Gamma_{c}\}$ for 
$C(g \Gamma_{c})$ using 
Theorem \ref{Thm:CompressedCentralizer}.  
Then $J$ is generated by
$\{w_{1},\ldots,w_{m-m_{c}},a_{c1},a_{c2},\ldots,a_{cm_{c}}\}$.
Compute $\coords{g^{-1}h^{v}}$ and use 
Theorem \ref{Thm:CompressedMembership} to determine if $g^{-1}h^{v}\in f(J)$ 
and if so use Theorem \ref{Thm:CompressedKernel} to find $w\in G$ such that 
$g^{-1}h^{v}=f(w)$.  Return $u=vw^{-1}$.

\emph{Complexity and length of $u$.} We first prove the length bound by induction on $c$.  For $c=1$, we have $u=1$.  
In the inductive case, each $w_{i}$ has length bounded by a degree 
$(16mc^{2})^{3cm}$ polynomial function of $L$ and $v$ has length 
bounded by a degree $(16mc^{3})^{m}$ polynomial function of $L$. Take 
$\delta=(16mc^{3})^{3cm}$, which bounds both of these degrees.
Then the input 
to Theorem \ref{Thm:KernelAndPreimage}, which consists 
of the pairs $(w_{i},[g,w_{i}])$ for $i=1,\ldots,m-m_{c}$, the 
pairs $(a_{cj},[g,a_{cj}])$ for $j=1,\ldots,m_{c}$, and the element $g^{-1}h^{v}$, 
also has length bounded by a degree $\delta$ polynomial function 
of $L$.  Hence $w$ has length bounded by a polynomial function of degree 
\[
(16mc^{3})^{3cm} \cdot 2m(8(c+1)^{3})^{m} < (16m(c+1)^{3})^{3(c+1)m}.
\]
Since this is greater than the degree bound for $v$, the output $u=vw^{-1}$ 
satisfies this degree bound.

Logarithmic space complexity follows immediately from the fact that the conjugator length only grows by a polynomial function of $L$ and 
the depth of the recursion is constant.  
The time complexity arises entirely from 
the computation of Mal'cev coordinates, in Theorem \ref{Thm:Centralizer} and 
of $g^{-1}h^{v}$.  Indeed, Theorems \ref{Thm:CompressedMembership} are each  
invoked with a constant number of inputs, each having bit-size $O(\log L)$, and therefore their time complexity is $O(\log^{3} L)$.
\end{proof}

We solve the compressed version of the conjugacy problem in the same way, 
using the compressed version of the centralizer algorithm.


\begin{theorem}
Let $G$ be a finitely generated nilpotent group.  There is an algorithm that, given two straight-line programs 
$\mathds{A}$ and $\mathds{B}$ over $G$ (or binary Mal'cev coordinate tuples), determines in time $O(L^{3})$ whether or not 
$\eval{\mathds{A}}$ and $\eval{\mathds{B}}$ are conjugate in $G$, where $L=|\mathds{A}|+|\mathds{B}|$. If so, a straight-line program over $G$ of size polynomial 
in $L$ producing a conjugating element is returned.
\end{theorem}

Our solution to the conjugacy probelm in nilpotent groups allows to take advantage of results in~\cite{Bumagin}, where Bumagin showed that a group $G$ hyperbolic relative to finitely generated subgroups $P_1,\ldots,P_k$ has polynomial-time solvable conjugacy (search) problem whenever the parabolic subgroups $P_1,\ldots,P_k$ have polynomial-time solvable conjgacy (search) problem. Together with Theorem~\ref{Thm:CP} this immediately gives the following result.

\begin{theorem}\label{th:rel_hyp_conj} Let $G$ be a finitely generated group hyperbolic relative to finitely generated nilpotent subgroups $P_1,\ldots, P_k$.
Then the word problem, the conjugacy problem, and the conjugacy search problem in $G$ can be solved in polynomial time. Moreover, the time complexity of the word and conjugacy problems is $O(L^3\log L)$, where $L$ is the length of input.
\end{theorem}
\begin{proof}
Note that~\cite[Theorem 1.5]{Bumagin} allows to obtain specific estimates on the time complexity of the above algorithms in terms of the complexity of word and conjugacy (search) problems in $P_1,\ldots,P_k$. In a given nilpotent group, on an input of length $L$, the time complexity of word problem is $O(L)$ by~\cite{HR}, 
and that of a conjugacy problem is $O(L\log^2 L)$ by Theorem~\ref{Thm:CP}. By~\cite[Theorem 1.5]{Bumagin}, we obtain that the time complexity of conjugacy problem in $G$ is bounded by
\[
T(L)=\max\{O(L\log^2L), O(L^2\cdot L\cdot \log L)\}=O(L^3\log L),\]
where $L$ is the length of the input.

The corresponding estimate in the case of conjugacy search problem, by~\cite[Theorem 5.10]{Bumagin} is given by 
\[
\max\{T(L), O(C_s(L))\},
\]
where $T(L)$ is as above, and $C_s(L)$ is a bound on the time complexity of conjugacy search problem in all parabolic subgroups $P_i$.
In our case, that gives a polynomial of a higher degree, according to Remark~\ref{re:binary_output}.
\end{proof}

\section{Presentation-uniform algorithms}
\label{Section:Uniform}

The algorithms presented in the previous sections do not include the nilpotent group $G$ as a part of the input. We now consider problems 
which do take $G$ as a part of the input.  Let $\mathcal{N}_{c}$ be the class of nilpotent groups of nilpotency class at most~$c$.  We consider groups in $\mathcal{N}_{c}$ 
presented using at most a fixed number, $r$, of generators. In the algorithms in this section, the integers $c$ and $r$ are constants. First, we use 
Theorem~\ref{Lem:Compute_std_form} to find consistent nilpotent presentations for 
such groups.

\begin{proposition}\label{pr:find_presentation_poly}
Let $c$ and $r$ be fixed integers. There is an algorithm that, given a finite presentation 
with $r$ generators of a group $G$ in $\mathcal{N}_{c}$, 
produces a consistent nilpotent presentation of $G$ and an explicit isomorphism, in 
space logarithmic in the size $L$ of the given presentation.  
Further, the size of the presentation is bounded by a polynomial function of $L$ and if 
binary numbers are used in the output relators then the algorithm runs in time 
$O(L\log^{3} L)$.
\end{proposition}
\begin{proof}
Let $G$ be presented as $G=\langle X\mid R\rangle$. Let $F=F(X)$ be the free nilpotent group of class $c$ on generators $X$. As a precomputation, produce a consistent nilpotent presentation $\langle Y\mid S\rangle$ for $F$. One may do this in such a way that 
$X\subset Y$ and elements of $Y\setminus X$ are iterated commutators (so-called `basic commutators') of elements of $X$.  Consider the natural surjection $\phi:F\rightarrow G$ 
and let $N=\ker(\phi)$, which is the normal closure of $R$ in $F$.  
Denoting $R=\{r_1,\ldots, r_k\}$, $N$ is generated  by iterated commutators $[\ldots[[r_i,x_1],x_2],\ldots,x_j]$, where $i=1,\ldots,k$, $j\le c$, and $x_1,\ldots, x_j\in X\cup X^{-1}$.
The total length of these generators is linear in $L$ since $c$ and $r$ are constant.  We now produce this generating set 
and apply Theorem~\ref{Lem:Compute_std_form} in $F$ with this set, producing 
the full-form sequence $T$ for $N$.

Now $G\simeq \langle Y\mid S\cup T\rangle$, and we claim that this is a consistent nilpotent presentation.  Since $\langle Y\mid S\rangle$ is a nilpotent 
presentation and the elements of $T$ add relators of the form (\ref{stdpolycyclic1}), 
the presentation is nilpotent. To prove that it is consistent, suppose some 
$y_i\in Y$ has order $\alpha_i$ modulo $\langle y_{i+1},\ldots,y_m\rangle$ in 
$\langle Y\mid S\cup T\rangle$. Since the order is infinite in $F$, there must 
be element of the form $y_i^{\alpha_i}y_{i+1}^{\alpha_{i+1}}\cdots y_m^{\alpha_m}$ in 
$N$. But then, by Lemma~\ref{Lem:UniqueStandardForm}, $T$ must contain an element $y_i^{\alpha'_i}y_{i+1}^{\alpha'_{i+1}}\cdots y_m^{\alpha'_m}$ where $\alpha'_i$ divides $\alpha_i$.  Hence $\alpha_i$ cannot be smaller than $\alpha'_i$ 
and so the presentation is consistent.

The space and time complexity, and the polynomial bound on the size of the presentation, follow immediately from 
Theorem \ref{Lem:Compute_std_form} since $c$ is fixed and $m$ depends only on $c$ and $r$.
\end{proof}

If $G$ is restricted to the class $\mathcal{N}_{c}$ and presented with a fixed number 
$r$ of generators, the above result can be employed to produce presentation-uniform versions of our algorithms for problems (\ref{Prob:NF})-(\ref{Prob:CP}) (see list on p.~\pageref{Prob:NF}) that run in logarithmic space and quasi-linear time.

\begin{theorem}\label{co:uniform_polytime}
Let $\Pi$ denote any of the problems (\ref{Prob:NF})--(\ref{Prob:CP}).
For all $c,r\in\mathbb N$, there is an algorithm that, given a finite presentation $\langle X| R\rangle$ with $|X|\leq r$ of a group in 
$\mathcal{N}_{c}$ and input of $\Pi$ as words over $X$, solves $\Pi$ in $\langle X| R\rangle$ on that input in logarithmic space and quasi-linear time.  
In the case of (\ref{Prob:KP}), the second 
group may be specified in the input by a presentation but must also be in $\mathcal{N}_{c}$ and use $r$ generators.
\end{theorem}
\begin{proof}
Let $L$ denote the total input size. 
By Proposition~\ref{pr:find_presentation_poly}, we can produce a consistent nilpotent presentation 
for~$G$, of size polynomial in $L$, in logarithmic space and time 
$O(L\log^3 L)$. 
The number of generators $m$ of this presentation depends only on $c$ and $r$, 
as it is equal to the number of generators of the precomputed nilpotent presentation 
of the free nilpotent group $F$ of class $c$.

We compute, in advance, the functions $p_{i}$ and $q_{i}$ for any nilpotent presentation 
on $m$ generators with the exponents $\alpha_{ijl},\beta_{ijl}, \mu_{il}, e_{i}$ appearing as variables (see 
the discussion following Lemma~\ref{le:malcev}). Substituting the correct values for 
these exponents (obtained from the nilpotent presentation of $G$) we obtain the functions $p_{i}$ and $q_{i}$ for the specific 
group $G$.  Note that the magnitude of all the numbers $\alpha_{ijl},\beta_{ijl}, \mu_{il}, e_{i}$ is bounded by a polynomial function 
of $L$, hence they may be encoded as $O(\log L)$-bit numbers.

We note that Theorem~\ref{le:poly_coords} is satisfied for the obtained presentation, by Remark~\ref{Rem:LowerCentral}. The constant $\kappa$ depends on the presentation 
of $G$, but may be seen from the proof of Theorem~\ref{le:poly_coords} to be 
bounded by a polynomial function of $L$ of degree depending on $c$ and $r$.  Hence 
the magnitude of each coordinate of a word $w$ is bounded by a polynomial function 
of $L$ of constant degree.  One may substitute this bound in subsequent arguments 
in place of $\kappa |w|^{i}$ without affecting the logspace or quasi-linear time 
nature of the computations.  In Lemma \ref{Lem:MatrixBound}, the constants $C$ and 
$K$ will be greater, but remain bounded by a constant depending on $c$ and $r$.

Specifically, the algorithms of Theorems 
\ref{th:compute_malcev}, \ref{Lem:Compute_std_form}, 
\ref{th:logspace_membership}, \ref{Thm:EffectiveCoherence}, \ref{Thm:KernelAndPreimage}, 
\ref{Thm:Centralizer}, \ref{Thm:CP} as well as Corollaries \ref{cor:poly_membership} 
and \ref{Cor:NonBinaryCoherence} run in logspace and time $O(L\log^3 L)$ 
when $G$ (from class $\mathcal{N}_{c}$ with $r$ generators) is included in the 
input.  The length bounds given in these results do not hold as stated, but do 
hold with polynomials of some higher (but still constant for fixed $c$ and $r$) degree.


\end{proof}

\bibliographystyle{amsplain}
\bibliography{nilpotent_bibliography}

\end{document}